\documentclass[10pt,reqno]{amsart}
\usepackage{mathtools}
\mathtoolsset{showonlyrefs}
 \allowdisplaybreaks 
\usepackage[english]{babel}
\usepackage{ragged2e}
\usepackage[utf8]{inputenc}
\usepackage[T1]{fontenc}
\usepackage{amsmath}
\usepackage{amsfonts}
\usepackage{nccmath}

\usepackage{etoolbox}
\makeatletter
\pretocmd{\tableofcontents}{%
  \begingroup
  \setlength{\parskip}{0pt}%
  \setlength{\baselineskip}{15pt}%
}{}{}

\apptocmd{\tableofcontents}{%
  \endgroup
}{}{}
\makeatother

\makeatletter
\renewcommand{\tocsection}[3]{%
  \indentlabel{\@ifnotempty{#2}{\bfseries\ignorespaces#1 #2\quad}}\bfseries#3}
\renewcommand{\tocsubsection}[3]{%
  \indentlabel{\@ifnotempty{#2}{\ignorespaces#1 #2\quad}}#3}

\newcommand\@dotsep{4.5}
\def\@tocline#1#2#3#4#5#6#7{\relax
  \ifnum #1>\c@tocdepth 
  \else
    \par \addpenalty\@secpenalty\addvspace{#2}%
    \begingroup \hyphenpenalty\@M
    \@ifempty{#4}{%
      \@tempdima\csname r@tocindent\number#1\endcsname\relax
    }{%
      \@tempdima#4\relax
    }%
    \parindent\z@ \leftskip#3\relax \advance\leftskip\@tempdima\relax
    \rightskip\@pnumwidth plus1em \parfillskip-\@pnumwidth
    #5\leavevmode\hskip-\@tempdima{#6}\nobreak
    \leaders\hbox{$\m@th\mkern \@dotsep mu\hbox{.}\mkern \@dotsep mu$}\hfill
    \nobreak
    \hbox to\@pnumwidth{\@tocpagenum{\ifnum#1=1\bfseries\fi#7}}\par
    \nobreak
    \endgroup
  \fi}
\AtBeginDocument{%
\expandafter\renewcommand\csname r@tocindent0\endcsname{0pt}
}
\def\l@subsection{\@tocline{2}{0pt}{2.5pc}{5pc}{}}
\makeatother

\usepackage{amssymb}
\usepackage{graphicx}
\usepackage{bbm}
\usepackage{amsmath,amsfonts,amsthm,amssymb,amscd, appendix, verbatim}
\usepackage{xpatch}
\usepackage{parskip}

\usepackage{setspace, color}
\usepackage{array}
\usepackage{cite}
\usepackage{hyperref}

\numberwithin{equation}{section}

\newcommand{\R}{\mathbb{R}}

\newcommand{\Z}{\mathbb{Z}}

\newcommand{\Sc}{\mathcal{S}}

\renewcommand{\k}{\kappa}

\newcommand{\M}{\mathcal{M}}

\newcommand{\weakto}{\rightharpoonup}

\newcommand{\inn}[2]{\langle #1, #2 \rangle}

\usepackage{anysize}
\usepackage{lineno}
\DeclareMathSymbol{\shortminus}{\mathbin}{AMSa}{"39}

\marginsize{2cm}{2cm}{1.5cm}{1.5cm}

\renewcommand{\Re}{\operatorname{Re}}

\newtheorem{teo}{Theorem}[section]
\newtheorem{prop}[teo]{Proposition}
\newtheorem{lem}[teo]{Lemma}
\newtheorem{re}[teo]{Remark}
\newtheorem{de}[teo]{Definition}
\newtheorem{cor}[teo]{Corollary}

\numberwithin{equation}{section}
\title[Large data scattering for the $k$-dgBO in the energy space]{Large data scattering for the defocusing $k$-dispersion generalized Benjamin-Ono equation in the energy space}

\date{}
\author{
Luccas Campos, 
Felipe Linares, and
Thyago S. R. Santos 
}
\begin{document}

\begin{abstract}
\noindent We study the defocusing $k$-dispersion generalized Benjamin-Ono equation. 
For every even integer $k\geq 4$, we prove that solutions with initial data in the energy space $H^{\frac{\alpha}{2}}$ are global in time and scatter. 
The proof combines the concentration-compactness-rigidity method of Kenig and Merle with techniques based on the Caffarelli-Silvestre extension and Tao's monotonicity formula adapted to the fractional dispersion setting.
\end{abstract}
\maketitle

\tableofcontents

\section{Introduction}
We consider the Cauchy problem for the defocusing $k$-dispersion generalized Benjamin-Ono equation
\begin{equation}\label{DGBO}
    \partial_t u + D^\alpha \partial_x u + \partial_x(u^{k+1}) = 0,
\qquad (t,x)\in \R\times\R, \tag{k-dgBO}
\end{equation}
where $u=u(t,x)$ is real-valued, $\alpha\in(1,2]$, and $k$ is an even positive integer. The operator $D^\alpha$ is the fractional derivative defined by the Fourier multiplier
\begin{equation}\label{eq:fractional_laplacian}
D^\alpha f(x):=\left(|\xi|^\alpha \widehat f(\xi)\right)^\vee(x).
\end{equation}
The \eqref{DGBO} family includes two fundamental one-dimensional dispersive models. When $\alpha=2$, it becomes the defocusing generalized Korteweg-de Vries equation, which, in the particular case $k=1$, corresponds to the classical KdV equation. At the formal endpoint $\alpha=1$, it becomes the generalized Benjamin-Ono equation, where $D\partial_x=\mathcal H\partial_x^2$ and
\begin{equation}\label{hilberttranform}
(\mathcal H f)(x):=(-i\operatorname{sgn}(\xi)\widehat f(\xi))^\vee(x).
\end{equation}
This class of equations originates in the theory 
shallow-water waves. We refer the reader, for instance, to the classical works of Korteweg and de Vries, Benjamin, and Ono \cite{KdV1895,Benjamin1967,Ono1975}, as well as to the general derivation by Shrira and Voronovich of the dgBO dispersion \cite{Shrira_Voronovich_1996} in the context of weakly nonlinear long vorticity waves in the coastal zone where the large-scale cross-shore depth depends algebraically on the offshore distance. 

For genuinely fractional dispersion, however, the equation is no longer governed by the local derivatives available in the KdV case, nor by the explicit Hilbert transform identities that play a central role in the Benjamin-Ono theory. For instance, in contrast to the classical KdV and Benjamin–Ono equations, which are completely integrable, no integrable structure is known for the genuinely fractional regime  $\alpha \in (1,2)$.
Nevertheless, for sufficiently regular real-valued solutions, the \eqref{DGBO} equation conserves the mass
\begin{equation}\label{eq:mass}
M[u](t):=\int_\R u^2(t,x)\,dx=M[u_0]
\end{equation}
and the energy
\begin{equation}\label{eq:energy}
E[u](t):=\frac12\int_\R |D^{\frac{\alpha}{2}}u(t,x)|^2\,dx
+\frac{1}{k+2}\int_\R u^{k+2}(t,x)\,dx=E[u_0].
\end{equation}
The equation also has the scaling symmetry
\begin{equation}\label{scale}
u_\lambda(t,x):=\lambda^{-\frac{\alpha}{k}}
u\left(\lambda^{-(\alpha+1)}t,\lambda^{-1}x\right),\qquad \lambda>0,
\end{equation}
which leaves invariant the homogeneous Sobolev norm $\dot H^{s_c}(\R)$ with
\begin{equation}\label{critical-index}
s_c:=\frac12-\frac{\alpha}{k}.
\end{equation}
Hence, the energy space $H^{\frac{\alpha}{2}}(\R)$ lies above the scaling-critical regularity for the range $\alpha>1$ and $k\geq4$ considered in this paper.

The local and global well-posedness theory for generalized Benjamin-Ono and fractional KdV-type equations has been developed in a number of works (see, \textit{e.g.},  \cite{MR1033618,MR1122309,MolinetRibaud2004,MR2859931,MR3200754,LinaresPilodSaut2014,Molinet-Tanaka-2510.27461}). In particular, in \cite{CamposLinaresSantos2024} the authors established sharp well-posedness and small-data scattering for \eqref{DGBO} at the scaling-critical regularity. More precisely, if $\|u_0\|_{\dot H^{s_c}}$ is sufficiently small, then the corresponding solution is global and there exist $u_0^+,u_0^-\in \dot H^{s_c}(\R)$ such that
\begin{equation*}
\lim_{t\to\pm\infty}\|u(t)-V_\alpha(t)u_0^\pm\|_{\dot H^{s_c}}=0,
\end{equation*}
where $V_\alpha(t)$ denotes the linear group generated by $-D^\alpha\partial_x$. Equivalently, scattering is implied by the finiteness of the critical spacetime norm
\begin{equation*}
\|u\|_{L_x^{p_{s_c}}L_t^{q_{s_c}}(\R\times\R)}
\end{equation*}
for the critical admissible exponents $(p_{s_c},q_{s_c})$ dictated by the scaling (cf. Definition \ref{de:functional_spaces}).

The purpose of this work is to waive the smallness requirement in order to prove scattering for arbitrary data in the energy space. This is a substantially stronger statement than merely global well-posedness. Indeed, conservation of energy prevents the growth of the energy norm, but scattering requires additional spacetime integrability: the time decay of the solution must be strong enough so that the nonlinear term is but perturbative in the limit $t\to\pm\infty$. Thus, the problem asks whether the \textit{a priori} bounds supplied by the conservation laws can be bootstrapped into a global-in-time dispersive estimate.


There are strong reasons to expect an affirmative answer in the defocusing case. For instance, the potential part of the energy 
does not exhibit the variational obstruction associated with focusing ground states. The evenness of $k$ is essential here: it guarantees the definite favorable sign in the potential energy and also underlies the positivity mechanisms used later in the rigidity argument. The linear group $V_\alpha(t)$ for $\alpha \in (1,2]$ provides enough dispersion to produce useful Strichartz and maximal-function estimates, while the monotonicity already observed in the defocusing gKdV and gBO settings is unable to harbor a compact coherent object indefinitely. Heuristically, if scattering were false, there would have to exist a minimal solution for which dispersion and nonlinearity remained exactly balanced for all times. Such a solution would behave like a localized, almost periodic object. The role of the rigidity argument is to prove that the monotonicity induced by the combined effects of the dispersion and the defocusing nonlinearity renders this minimal counterexample impossible.

Large-data scattering is one of the central problems in the modern theory of nonlinear dispersive equations precisely because this heuristic cannot be achieved by conservation laws alone. Early scattering results relied on Lin-Strauss Morawetz estimates \cite{LS1978}. Later on, Bourgain's induction-on-energy ideas  \cite{B98} and the concentration-compactness-rigidity method of Kenig and Merle \cite{KENIG,KENIG2,MR3059764}, rooted in earlier profile decompositions such as the one by Bahouri and Gérard \cite{BahouriGerard1999}, provided a robust way to turn the above contradiction scenario into a proof. In nonlinear Schr\"odinger and wave equations, this collection of ideas led to several landmark results, including works of Bourgain \cite{MR1691575}, Duyckaerts, Holmer and Roudenko \cite{MR2470397}, Killip, Tao and Visan \cite{KTV}, Killip and Visan \cite{KiVi}, and Dodson \cite{dodson2,dodson3}. For one-dimensional dispersive models with derivative nonlinearities, Dodson proved scattering for the mass-critical defocusing gKdV equation \cite{MR3625190}, Farah, Linares, Pastor and Visciglia obtained large-data scattering for the defocusing supercritical gKdV equation \cite{MR3772197}, and Kim and Kwon in \cite{MR4009456} treated the defocusing generalized Benjamin-Ono equation in the energy space $H^{\frac12}(\R)$.

Our main result extends this large-data scattering picture to the genuinely fractional dispersive family \eqref{DGBO}.

\begin{teo}\label{teo:scat}
Let $\alpha\in(1,2]$ and let $k\geq4$ be even. For every $u_0\in H^{\frac{\alpha}{2}}(\R)$, the solution to \eqref{DGBO} with initial data $u_0$ is global in time and scatters in $H^{\frac{\alpha}{2}}(\R)$ both forward and backward in time. That is, there exist $u_0^+,u_0^-\in H^{\frac{\alpha}{2}}(\R)$ such that
\begin{equation*}
\lim_{t\to\pm\infty}\|u(t)-V_\alpha(t)u_0^\pm\|_{H^{\frac{\alpha}{2}}}=0.
\end{equation*}
\end{teo}

While the concentration-compactness machinery already requires a careful adaptation to the fractional dispersive setting, the rigidity argument presents a further challenge.
In the KdV setting, Tao's monotonicity method \cite{MR2276483} exploits the local structure of the operator $\partial_x^3$. In the Benjamin-Ono setting, Kim and Kwon \cite{MR4009456} use identities and commutator formulas tied to the Hilbert transform $\mathcal H$. For $D^\alpha\partial_x$, with $1<\alpha<2$, neither type of structure is available in the same form. Localizing mass or energy produces commutators with the non-polynomial multiplier $|\xi|^\alpha$. These commutators are substantially more delicate than their Hilbert transform analogues, providing much less smoothness \cite{KatoPonce1988,Li2019}. Moreover, the absence of a fractional version of the Leibniz rule (as an identity) prevents one from freely integrating by parts inside localized virial or Morawetz functionals.

A key point of our argument is that the Caffarelli-Silvestre extension \cite{CS07} restores enough locality to the problem to recover the sign information needed for rigidity. By representing $D^\alpha$ as a Dirichlet-to-Neumann operator for a degenerate elliptic extension in one more variable, we obtain localized identities and integration-by-parts formulas that are not accessible from the Fourier multiplier alone. In particular, this viewpoint is used to establish the nontrivial semidefiniteness
\begin{equation*}
\int_\R u^{k+1}D^\alpha u\,dx\geq 0, 
\end{equation*}
which is consistent with the convexity principle for fractional Laplacians \cite{CordobaCordoba2004} and is paramount to the monotonicity formula. Moreover, the fractional extension framework enables us to pass to a local almost-non-negativity estimate, which is then used to preclude the nonscattering counterexample. This is one of the main new features of the argument: the extension method allows the fractional operator to be handled as a local object at the places where locality facilitates the rigidity proof.

The proof is organized around the concentration-compactness-rigidity strategy. First, one needs a global theory that can ascertain scattering through a spacetime norm, which must guarantee small-data scattering and remain stable under small perturbations. This theory started its development in \cite{CamposLinaresSantos2024} and is adapted to our context in Sections \ref{sec-2} and \ref{sec-3}, together with the notation and preliminary estimates used throughout the paper. After that, one proves a linear profile decomposition adapted to the group $\{V_\alpha(t)\}_{t\in\R}$. We show that any bounded sequence of data in the energy space can be split, modulo subsequence, into elementary linear profiles that are asymptotically orthogonal, plus a remainder that is negligible in the critical spacetime norms. 

Combining the first two steps hints at a contradiction argument. If Theorem \ref{teo:scat} were false, then among all nonscattering solutions, one wishes to produce a minimal obstruction, usually called a ``critical solution''. This solution is necessarily global, lies at the smallest possible nonscattering size, and by the profile decomposition, wave operators and perturbation theory, can be shown to have an orbit that is compact modulo the symmetries of the equation. Intuitively, it is the only kind of object that could prevent scattering: it cannot radiate mass or energy away, nor split into more than a single nonlinear profile, because then every part would scatter and could be removed by the perturbation theory. Section \ref{sec-6} is devoted to making this compactness mechanism precise.

The remaining sections of the paper prove that this compact object behaves too nicely to be allowed to exist. It is imperative that the fractional nature of the operator be addressed first: Section \ref{sec-4} develops the identities for $D^\alpha$ and the Caffarelli-Silvestre extension needed to recover local information from a nonlocal operator. With these tools, Section \ref{sec-5} establishes the monotonicity formula adapted to \eqref{DGBO}. Finally, Section \ref{sec-7} combines the compactness of the critical element and this monotonicity formula. The compactness guarantees that a fixed amount of the solution remains trapped together in space, while monotonicity forces a one-way separation of the mass and energy centers. Alas, these two conclusions are incompatible. This contradiction rules out the critical element and completes the proof of scattering.

\begin{re}
The use of the Caffarelli-Silvestre extension in this setting provides a flexible way to extract local information from fractional dispersive operators. We expect this viewpoint to be useful in other rigidity arguments where the main obstruction is the lack of local virial identities.
\end{re}

\begin{re}
The argument here is different from the one in \cite{MR4009456}. Kim and Kwon work with the Benjamin-Ono dispersion $\mathcal H\partial_x^2$, whose commutator structure is more explicit. In contrast, the estimates for $D^\alpha\partial_x$ require a genuinely fractional analysis: commutators with $D^\alpha$ are less algebraic, localized integrations by parts are more restricted, and the positivity needed for monotonicity must be recovered through the extension method.
\end{re}

\begin{re}
The corresponding large-data scattering problem for odd values of $k$ remains open for gKdV, gBO, and for the fractional family considered here.
\end{re}



\noindent\textbf{Acknowledgments.}
L. Campos was partially supported by CNPq grants 07733/2023-8 and 404800/2024-6, the CAPES/Cofecub grant 88887.879175/2023-00 and the FAPEMIG grants APQ-03186-24 and APQ-03752-25. F. Linares was partially supported by CNPq grant 310329\!/\!2023-0 and FAPERJ grant E\--26\!/\!200.236\!/\!2026. T.S.R. Santos was partially supported by FAPESP Grant No.2024/15587-1. The author T.S.R. Santos also gratefully acknowledges the Universidade Federal de Minas Gerais for its hospitality during the period in which part of this work was carried out.

\section{Notations and preliminaries}\label{sec-2}
For any two nonnegative quantities $X$ and $Y$, the notation $X\lesssim Y$ implies the existence of an absolute constant $C>0$, independent of $X$ and $Y$, such that $X \leq CY$. We will denote $X\sim Y$ when both $X \lesssim Y$ and $Y \lesssim X$. The notation $X \ll Y$ is used when $X \gtrsim Y$ does not hold.

\medskip
We begin by recalling the definition of Sobolev spaces via Fourier multipliers.

\begin{de}
For $u \in \Sc(\R)$, where $\Sc(\R)$ denotes the Schwartz class, and $\beta \in \R$, we define the \textit{Bessel potential of order $-\beta$} by the Fourier multiplier
$$
J^{\beta} u (x):= \left( \langle \xi \rangle^{\beta} \widehat{u}(\xi)\right)^\vee(x),
$$
where, for $x \in \R$, $\langle x \rangle := \sqrt{1 + x^2}$. Here, we use the following definition for the Fourier transform:
$$
\widehat{f}(\xi) := \int_\R e^{-i x \cdot \xi} f(x)\, dx.
$$
We denote by $H^{\beta,p}(\R)$ the usual (inhomogeneous) Sobolev space, equipped with the norm
$$
\|u\|_{H^{\beta,p}} = \|J^{\beta} u\|_{L^{p}}.
$$
Likewise, for $1 \leq p \leq \infty$, we denote by $\dot{H}^{\beta,p}(\R)$ the homogeneous Sobolev space, equipped with the norm
\begin{equation}\label{sobolev homogeneo}
    \|u\|_{\dot{H}^{\beta,p}} = \|D^{\beta} u\|_{L^{p}},
\end{equation}
where $D^{\beta}$ was defined in \eqref{eq:fractional_laplacian}. In the case $p=2$, we omit the superscript $p$.
\end{de}

\medskip
We then recall some classical inequalities that will be repeatedly used in the sequel.

\begin{lem}[Sobolev embedding, see \cite{Grafakos}]
Let $u \in H^\beta(\mathbb{R})$, $\beta > 1/2$. Then, for any $2 \leq p \leq \infty$,
\begin{equation}
    \|u\|_{L^p} \lesssim \|u\|_{H^\beta}.
\end{equation}
\end{lem}
\medskip
The following chain rule estimate controls nonlinearities in Sobolev spaces.

\begin{lem}[Fractional chain rule, see  \cite{Grafakos}]
Let $k$ be a positive integer, $\beta>1/2$, and $u \in H^\beta(\mathbb{R})$. Then
\begin{equation}
    \|u^k\|_{H^\beta} \lesssim \|u\|_{L^\infty}^{k-1}\|u\|_{H^\beta}.
\end{equation}
\end{lem}

\medskip
We also make use of Littlewood-Paley frequency projections.

\begin{de}
Let $\chi(\xi)$ be a radial bump function supported in the ball $\{\xi \in \mathbb{R}^d : |\xi| \leq \sqrt2\}$ and equal to $1$ on the ball $\{\xi \in \mathbb{R}^d : |\xi| \leq 1\}$. For each dyadic number $N \in 2^{\mathbb Z}$, we define the Littlewood-Paley projections by the Fourier multipliers
\begin{equation}
P_{\leq N} f(\xi) := \chi(\xi/N)\widehat{f}(\xi) \quad \text{and} \quad
P_{> N} f(\xi) := \bigl(1 - \chi(\xi/N)\bigr)\widehat{f}(\xi).
\end{equation}
\end{de}

\medskip
Associated to these frequency projections, we recall the standard Bernstein inequalities.

\begin{lem}[Bernstein inequalities, see \cite{livroTao}]
For any $1 \leq p \leq q \leq \infty$, $\beta \in \mathbb{R}$, and $N \in 2^{\mathbb Z}$,
\begin{align}
\| P_{\leq N} f \|_{L^q} &\lesssim \| f \|_{L^p}, \\
\| P_{\leq N} f \|_{L^q} &\lesssim N^{\frac{1}{p} - \frac{1}{q}} \| P_{\leq N} f \|_{L^p}, \\
\|  P_{> N} f \|_{L^p} &\sim N^{-\beta} \|D^\beta P_{> N} f \|_{L^p}.
\end{align}
\end{lem}

\medskip
We also recall the definition of the Hardy-Littlewood maximal operator, which plays an important role in estimates related to the Caffarelli-Silvestre fractional extension.

\begin{de}\label{fun. Maximal}
For $f \in L^{1}_{\mathrm{loc}}(\mathbb{R})$, we define the Hardy-Littlewood maximal operator by
$$
\mathcal{M}f(x) := \sup_{B \ni x} \frac{1}{|B|} \int_{B} |f(y)|\, dy,
$$
where the supremum is taken over all balls $B \subset \mathbb{R}$ containing $x$.
\end{de}
\medskip
Next, we introduce the mixed space-time Lebesgue norms that will be used throughout the paper, together with the corresponding Strichartz estimates.

\begin{de}
For a time interval $I \subset \R$, and $1\leq p,q\leq \infty$, we define the mixed Lebesgue spaces 
$L^p_x L^q_t (I) = L^p_x(\R:L^q_t(I))$ by the norm
\begin{equation}
    \|f\|_{L^p_x L^q_t (I)} = \left\| \|f\|_{ L^q_t (I)}\right\|_{L^p_x(\R)}.
\end{equation}
The mixed spaces $\dot H^{\beta,p}_x L^q_t$ and $H^{\beta,p}_x L^q_t$ are defined by the norms
\begin{equation}
    \|f\|_{\dot H^{\beta,p}_x L^q_t (I)} := \left\| \|D^\beta_xf\|_{ L^q_t (I)}\right\|_{L^p_x(\R)},\quad\text{and}\quad\|f\|_{H^{\beta,p}_x L^q_t (I)}: = \left\| \|J^\beta_xf\|_{ L^q_t (I)}\right\|_{L^p_x(\R)}.
\end{equation}
\end{de}

We also present the functional spaces repeatedly used throughout the text.

\begin{de}\label{de:functional_spaces}
    
For $\alpha>1$, we fix the indices
\begin{equation}
    p_0 = \frac{4\alpha+2}{\alpha-1}, \quad\text{and}\quad q_0 = \frac{4\alpha+2}{3},
\end{equation}
and consider the resolution space
\begin{equation}
   {X}(I) =  
    C_t  H^{\frac\alpha2}_x (I \times \mathbb R) \cap  H^{\frac{\alpha+1}{2},p_0}_xL^{q_0}_t (I \times \mathbb R),
\end{equation}
endowed with the canonical norm. We also fix the following indices for $\alpha>1$, $k \geq 4$, 
 $$
 {p_{s_c}}= \frac{(2\alpha+1)k}{\alpha+2},\quad\text{and}\quad
 {q_{s_c}}= \frac{(2\alpha+1)k}{2(\alpha-1)},
 $$
and consider the \textit{scattering space}
\begin{equation}
 S(I) = L^{ p_{s_c}}_xL^{ q_{s_c}}_t(I \times \mathbb R) \cap \dot H^{s_c+\frac{1}{2},p_0}_xL^{q_0}_t (I \times \mathbb R).
\end{equation} 
To estimate the nonlinearity and forcing terms, we consider, for $\alpha>1$ and $s\in \R$, the homogeneous and inhomogeneous mixed spaces
\begin{equation}
\dot N ^s(I) =  \dot H^{s+\frac{1}{2},\frac{4\alpha+2}{3(\alpha+1)}}_x L^{\frac{4\alpha+2}{4\alpha-1}}_t (I \times \mathbb R)\quad\text{and}\quad  N ^s(I) =  H^{s+\frac{1}{2},\frac{4\alpha+2}{3(\alpha+1)}}_x L^{\frac{4\alpha+2}{4\alpha-1}}_t (I \times \mathbb R).
\end{equation}
In particular, for the perturbation theory, we typically consider $s = \frac{\alpha}{2}$ or $s = s_c$. 
\end{de}

\begin{de}[Linear propagator]
    For $\alpha>1$, we let $\{V_{\alpha}(t)\}_{t \in \R}$ denote the linear propagator
associated with~\eqref{DGBO}, that is, for $f \in \mathcal{S}(\R)$
\begin{equation}\label{free evolution group}
\bigl(V_{\alpha}(t)f\bigr)(x) := \left( e^{-it|\xi|^\alpha\xi} \widehat{f}(\xi) \right)^\vee(x).
\end{equation}
\end{de}

\begin{lem}[Strichartz Estimates {\cite[Theorem 2.2 and Lemma 2.3]{CamposLinaresSantos2024}}] The following estimates hold:
\begin{itemize}
\item \underline{Homogeneous estimates:}
\begin{align}\label{strichartz_X}
\|V_\alpha(t)f\|_{X(I)} &\lesssim \|f\|_{H^{\frac{\alpha}{2}}_x},\\ \label{strichartz_S}
\left\|V_\alpha(t)f\right\|_{S(I)} &\lesssim \|f\|_{\dot H^{s_c}_x}.
\end{align}

\item \underline{Inhomogeneous estimates:}
\begin{align}
\label{inhomogeneous_strichartz_X}\left\| \int_I V_\alpha(t-\tau) \partial_xg \, d\tau \right\|_{X(I)} &\lesssim \|g\|_{N^{\frac{\alpha}{2}}(I)},\\
\left\| \int_I V_\alpha(t-\tau) \partial_xg \, d \tau \right\|_{S(I)} &\lesssim \|g\|_{\dot N^{s_c}(I)}. \label{inhomogeneous_strichartz_S}
\end{align}
\end{itemize}
\end{lem}
\begin{lem}[Nonlinear estimates {\cite[Proposition 3.1]{CamposLinaresSantos2024}}] The following estimates hold:
\begin{align}
       \|f^k g\|_{N^{\frac{\alpha}{2}}(I)}& \lesssim \|f\|_{S(I)}^k \|g\|_{X(I)} + \|f\|_{S(I)}^{k-1}\|f\|_{X(I)}\|g\|_{S(I)},\\
       \|f^k g\|_{\dot N^{s_c}(I)} &\lesssim \|f\|_{S(I)}^k \|g\|_{S(I)}.
\end{align}
    
\end{lem}

\section{An overview of the well-posedness theory}\label{sec-3}
One of the main ingredients in the compactness-contradiction argument is the
perturbation theory, which ultimately stems from the local well-posedness of the
Cauchy problem. As in many other canonical nonlinear dispersive equations, the local
theory for \eqref{DGBO} relies fundamentally on local smoothing estimates. A complete
critical well-posedness result was established recently by the authors in \cite{CamposLinaresSantos2024}. Therefore, this section is devoted to recalling the critical and subcritical well-posedness
results for \eqref{DGBO} equation for $k \geq 4$, and we refer the reader to \cite{CamposLinaresSantos2024} for further details on the proofs.

We note, however, that we work here in $H^{\frac\alpha2}(\R)$, in the defocusing, energy-subcritical case and with even values of $k$. Therefore, the well-posedness theory and the coercivity of the energy \eqref{eq:energy} guarantee that all solutions are  uniformly bounded in $H^{\frac\alpha2}(\R)$, hence globally defined. Accordingly, we state the results of \cite{CamposLinaresSantos2024} adapted to our setting.


\begin{de}[Global solution for \eqref{DGBO}]\label{def gwp} Let $\alpha \in (1,2]$, $k\in \Z^+$ even, and $u_0 \in H^{\frac\alpha2}(\R)$. We say that $u \in C_tH^{\frac{\alpha}{2}}_x(\R \times \R$) is a (global) solution to \eqref{DGBO} with initial data $u_0$ if $u \in {X}(I) \cap S(I) $ for any compact interval $I\subset \R$ and satisfies the Duhamel formula for all $t \geq 0$:
\begin{equation}\label{eq:duhamel}
u(t) = V_\alpha(t) u_0 
  + \int_{0}^{t} V_\alpha(t-t^\prime) \, \partial_x(u^{k+1})(t^\prime)\, dt^\prime.
\end{equation}

\end{de}

The next result establishes global well-posedness of \eqref{DGBO}. It is a corollary of the well-posedness theory from \cite{CamposLinaresSantos2024} and the coercivity given by the energy \eqref{eq:energy} in the defocusing case.
\begin{teo}[Subcritical local theory in $H^{\frac\alpha2}$, {\cite[Theorem 1.8]{CamposLinaresSantos2024}}]\label{Teorema existencia critica} Let $\alpha \in (1,2]$ and $k \geq 4$ be even. Then, for any $u_0 \in H^{\frac\alpha2}(\mathbb{R})$, there exists a unique global solution $u$
to \eqref{DGBO} with initial data $u_0$. Moreover, the data-to-solution map $u_0 \mapsto u$ is locally uniformly continuous.

\end{teo}

The authors also showed that controlling the $S$-norms plays a key role in
understanding the long-time behavior of solutions to \eqref{DGBO}. More
specifically, we have:

\begin{teo}[Scattering criterion, {\cite[Proposition 1.5]{CamposLinaresSantos2024}}]\label{prop:scattering_criterion}
Let $u_0 \in H^{\frac\alpha2}(\R)$ and let $u$ be the solution to \eqref{DGBO} given by Theorem \ref{Teorema existencia critica}. If
\begin{equation*}
    \|u\|_{S(\R_+)} < + \infty,
\end{equation*}
then $u$ scatters forward in time in $H^{\frac\alpha2}(\R)$. Moreover, there exists a function $C: [0,+\infty) \to [0,+\infty)$ such that
\begin{equation}
\|u\|_{X(\R_+)} \leq C(\|u\|_{S(\R_+)}) \|u_0\|_{H^\frac \alpha 2}.
\end{equation}

The analogous statement holds for negative times.

\end{teo}
As a consequence, the scattering theory for small initial data follows.
\begin{cor}[Small data theory,{\cite[Corollary 2.6]{CamposLinaresSantos2024}}]\label{cor:critical_small_data} For $\alpha \in (1,2]$ and $k \geq 4$ there exists $\delta_{sd}:=\delta_{sd}(\alpha,k)>0$ such that, if $u_0 \in {H}^{\frac\alpha2}(\R)$ satisfies
\begin{equation}\label{time_well_posedness}
    \|u_0\|_{H^{\frac\alpha2}} < \delta_{sd},
\end{equation}
then the solution $u$ obtained in Theorem \ref{Teorema existencia critica} satisfies
\begin{equation}
    \|u\|_{X(\R)} \lesssim \|u_0\|_{H^{\frac \alpha 2}}\quad\text{and}\quad \|u\|_{S(\R)} \leq 2\|V_{\alpha}(t)u_0\|_{S(\R)}.
\end{equation}
Moreover, $u$ scatters forward and backward in ${H}^{\frac\alpha2}(\R)$.
\end{cor}

We also have asymptotic completeness.

\begin{prop}[Wave operators {\cite[Proposition 3.2]{CamposLinaresSantos2024}}]\label{prop:wave_operator}
Let $v \in H^{\frac\alpha2}(\R)$. There exist global solutions $u_\pm $ to \eqref{DGBO} such that
\begin{equation}
\|u_\pm(t)- V_\alpha(t)v\|_{H^{\frac\alpha2}} \to 0 \quad\text{as}\quad t \to \pm\infty.
\end{equation}
\end{prop}

This concludes the overview of the well-posedness
theory for \eqref{DGBO}. In the next sections, we make use of these results to
implement the perturbative and compactness tools needed for the proof of our
main results.

\section{The minimal counterexample}\label{sec-6}


The main goal of this section is to prove the following result.

\begin{prop}\label{prop:nontrivial_localization}

Suppose that the conclusion of Theorem \ref{teo:scat} does not hold. Then, there exists a solution $u_c$ to \eqref{DGBO}, a function $x: \R_+ \to \R$ and $\tilde R(u_c)>0$ independent of time such that 
	
	\begin{equation}\label{eq:nontrivial_localization}
		\inf_{t \geq 0}\; \int_{|x-x(t)|\leq \tilde R(u_c)} |u_c(t,x)|^{k+2} \, dx >0.
	\end{equation}
	
 \end{prop}



The proof of this result remounts to the work of Kenig and Merle for the Schr\"odinger 
and wave equations, and has since been adapted to a variety of dispersive models. 
It relies crucially on four fundamental pillars: small-data scattering, long-time perturbation theory, a linear profile decomposition, and asymptotic completeness. 

    The first of these implies the existence of a positive threshold  below which scattering must occur, while allowing for the possibility of nonscattering solutions for larger data. Relying on the perturbation theory, we construct a sequence of solutions whose initial data approaches this threshold, while asymptotically ceasing to enjoy finiteness of space-time norms. Such a sequence, however, has no \textit{a priori} obligation to converge, not even up to subsequences or symmetries. 
    Still, we yearn for compactness.
    By decomposing the initial data into concentration bubbles via the profile decomposition, we employ asymptotic completeness to pass to nonlinear profiles. The minimality of the threshold begets \textit{a fortiori} coalescence, since any positive amount of separation would lead to scattering. We can then safely pass to a convergent subsequence and prove the desired localization.


\begin{teo}[Linear profile decomposition in $H^{\frac{\alpha}{2}}$]
\label{thm:main}
Let $(u_n)_{n \geq 1}$ be a bounded sequence in $H^{\frac{\alpha}{2}}(\R)$. 
Then, after passing to a subsequence, there exist:
\begin{enumerate}
    \item A sequence of \emph{profiles} $(\phi^j)_{j \geq 1} \subset H^{\frac{\alpha}{2}}(\R)$,
    \item Sequences of \emph{time parameters} $(\tau_n^j)_{n \geq 1} \subset \R$ and 
          \emph{translation parameters} $(x_n^j)_{n \geq 1} \subset \R$ for each $j \geq 1$,
\end{enumerate}
such that for every integer $J \geq 1$, we have the decomposition
\begin{equation}
    u_n(x) = \sum_{j=1}^{J} V_{\alpha}(\tau_n^j) \phi^j(x - x_n^j) + W_n^J(x),
    \label{eq:decomp}
\end{equation}
with the following properties:

\begin{enumerate}
    \item {Asymptotic orthogonality of parameters:} For any $j \neq k$,
    \begin{equation}
        \lim_{n \to \infty} \left( |\tau_n^j - \tau_n^k| + |x_n^j - x_n^k| \right) = +\infty.
        \label{eq:orthog_params}
    \end{equation}
    
    \item {Asymptotic Pythagorean expansion:} For any fixed $J$, as $n \to \infty$,
    \begin{align}
        \|u_n\|_{L^2}^2 &= \sum_{j=1}^J \|\phi^j\|_{L^2}^2 + \|W_n^J\|_{L^2}^2 + o_n(1), \label{eq:L2_expansion} \\
        \|u_n\|_{\dot{H}^{\frac{\alpha}{2}}}^2 &= \sum_{j=1}^J \|\phi^j\|_{\dot{H}^{\frac{\alpha}{2}}}^2 + \|W_n^J\|_{\dot{H}^{\frac{\alpha}{2}}}^2 + o_n(1), \label{eq:Hs_expansion}\\
        E[u_n] &= \sum_{j=1}^J E[V_{\alpha}(\tau_n^j) \phi^j]+E[W^J_n] + o_n(1).\label{eq:energy_expansion}
    \end{align}
    where $o_n(1) \to 0$ as $n \to \infty$.
    
    \item {Smallness of the remainder in Strichartz norms:}
    \begin{equation}
        \lim_{J \to \infty} \left( \limsup_{n \to \infty} \|V_{\alpha}(t) W_n^J\|_{S(\R)} \right) = 0.
        \label{eq:smallness}
    \end{equation}
    
    \item {Weak convergence of profiles:} For each $k \geq 1$,
    \begin{equation}
        V_{\alpha}(\tau_n^k) u_n(\cdot + x_n^k) \weakto \phi^k \quad \text{weakly in } H^{\frac{\alpha}{2}}(\R).
        \label{eq:weak_conv}
    \end{equation}
    
\end{enumerate}
\end{teo}

The main tool used to prove the linear profile decomposition above is a reverse inequality.

\begin{lem}[Reverse Strichartz inequality]
\label{lem:inverse_strichartz}
Let $(v_n)$ be a sequence in $H^{\frac{\alpha}{2}}(\R)$ and let $A>0$ such that
\begin{equation}
    \sup_{n } \| v_n\|_{H^{\frac\alpha2}} \lesssim_A 1
\end{equation}
and
\begin{equation}
    \liminf_{n \to \infty} \|V_{\alpha}(t) v_n\|_{S(\R)} \geq \delta > 0.
\end{equation}
Then there exist sequences $(\tau_n)$, $(x_n) \subset \R$,
$\phi \in H^{\frac\alpha2}(\R)$, and $\eta>1$ with $\|\phi\|_{H^{\frac{\alpha}{2}}} \gtrsim_{A}\delta^{\frac{\alpha (2\alpha+1)\eta}{3(\alpha-1)(\alpha-2s_c)(\eta-1)}}$, such that up to a subsequence,
\begin{equation}
    V_{\alpha}(\tau_n) v_n(\cdot + x_n) \weakto \phi \quad \text{weakly in } H^{\frac\alpha2}(\R).
\end{equation}
\end{lem}

\begin{proof}
Recall that $P_{\leq N}$ and $P_{>N}$ are the Littlewood-Paley projections of frequencies $\leq N$ and $>N$, respectively. Note that, by Strichartz and Bernstein,
\begin{align}
    \|P_{>N}V_{\alpha}(t)v_n\|_{S(\R)} &\lesssim \|P_{>N}v_n\|_{\dot H^{s_c}} \lesssim_A N^{-(\frac\alpha2-s_c)},
\end{align}
and, by H\"older, Strichartz and Bernstein, choosing $1< \eta < 1+\frac{s_c}{3}$,
\begin{align}
    \|D^{s_c+\frac12}P_{\le N}V_{\alpha}(t) v_n\|_{L^{p_0}_xL^{q_0}_t} &
    \\&\hspace{-50pt}\leq \|D^{s_c+\frac{\alpha+1}{2}}\!P_{\leq N}V_{\alpha}(t) v_n\|_{L^{\infty}_x L^{2}_t}^{\frac{4\shortminus\alpha}{2\alpha+1}}
    \|D^{s_c+\frac{\alpha\shortminus1}{6}}\!P_{\leq N}V_{\alpha}(t) v_n\|_{L^{\frac{6}{\eta}}_{t,x}}^{\frac{3(\alpha\shortminus1)}{(2\alpha+1)\eta}}\|D^{\scriptscriptstyle s_c+\frac{\alpha\shortminus1}{6}\!}P_{\leq N}V_{\alpha}(t) v_n\|_{L^{\infty}_{t,x}}^{\frac{3(\alpha\shortminus1)(\eta\shortminus1)}{(2\alpha+1)\eta}}\\
    &\hspace{-50pt}\lesssim_A N^{s_c} \|V_{\alpha}(t)P_{\leq N}\! v_n\|_{L^{\infty}_{t,x}}^{\frac{3(\alpha\shortminus1)(\eta\shortminus1)}{(2\alpha+1)\eta}},
\end{align}
and 
\begin{align}
    \|P_{\leq N}V_{\alpha}(t) v_n\|_{L^{p_{s_c}}_xL^{q_{s_c}}_t} &\leq \|D^{-\frac{\alpha-1}{k}}P_{\leq N}V_{\alpha}(t) v_n\|_{L^k_x L^{\infty}_t}^{\frac{4-\alpha}{2\alpha+1}}
    \|D^{\frac{4-\alpha}{3k}}P_{\leq N}V_{\alpha}(t) v_n\|_{L^{\frac{3k}{2\eta}}_{t,x}}^{\frac{3(\alpha-1)}{(2\alpha+1)\eta}}    \|D^{\frac{4-\alpha}{3k}}P_{\leq N}V_{\alpha}(t) v_n\|_{L^{\infty}_{t,x}}^{\frac{3(\alpha-1)(\eta-1)}{(2\alpha+1)\eta}}\\
    &\lesssim_AN^{s_c}\|V_{\alpha}(t)P_{\leq N} v_n\|_{L^{\infty}_{t,x}}^{\frac{3(\alpha-1)(\eta-1)}{(2\alpha+1)\eta}}.
\end{align}







Therefore, we can choose $\tau_n$, $x_n$ and $N(\delta) =C\, \delta^{-\frac{2}{\alpha-2s_c}}$ for large $C>0$ such that
\begin{equation}
    |V_{\alpha}(\tau_n)P_{\leq N} v_n(x_n)|^{\frac{\alpha-1}{k+\alpha-1}} \gtrsim_A \left(\delta - N^{-(\frac\alpha2-s_c)}\right)N^{-s_c} \gtrsim \ \delta^{\frac{\alpha}{\alpha-2s_c}}
\end{equation}
Now, passing to a subsequence, there exists $\phi \in H^{\frac\alpha2}(\R)$ such that
\begin{equation}
    V_{\alpha}(\tau_n)v_n(\cdot + x_n) \rightharpoonup \phi
\end{equation}
weakly in $H^{\frac\alpha2}(\R)$. Moreover, another application of Sobolev and weak convergence gives
\begin{equation}
     \|\phi\|_{H^{\frac\alpha2}} \gtrsim  \|P_{\leq N} \phi\|_{L^\infty} \geq |P_{\leq N}\phi(0)| = \lim_n |V_{\alpha}(\tau_n)P_{\leq N} v_n(x_n)| \gtrsim_A \delta^{\frac{\alpha (2\alpha+1)\eta}{3(\alpha-1)(\alpha-2s_c)(\eta-1)}}.
\end{equation}
\end{proof}
\begin{proof}[Proof of Theorem \ref{thm:main}]  We construct the profiles iteratively. Set $W_n^0 = u_n$. For $j \geq 1$, define
\begin{equation}
    \delta_j = \limsup_{n \to \infty} \|V_{\alpha}(t) W_n^{j-1}\|_{S^{s_c}}.
\end{equation}

If $\delta_j = 0$, we stop the process and set all $\phi^k = 0$ for $k \geq j$. If $\delta_j > 0$, apply Lemma \ref{lem:inverse_strichartz} 
to the sequence $(W_n^{j-1})_n$ to obtain sequences $(\tau_n^j)_n$, $(x_n^j)_n$ and a profile 
$\phi^j \in H^{\frac\alpha2}(\R) \setminus \{0\}$ such that, as $n \to +\infty$,
\begin{equation}
    V_{\alpha}(\tau_n^j) W_n^{j-1}(\cdot + x_n^j) \weakto \phi^j \quad \text{weakly in } H^{\frac\alpha2}(\R).
    \label{eq:weak_lim_j}
\end{equation}

Define the remainder
\begin{equation}
    W_n^j = W_n^{j-1} - V_{\alpha}(-\tau_n^j)\phi^j(\cdot - x_n^j)
    \label{eq:remainder_def}
\end{equation}
so that, by \eqref{eq:weak_lim_j},
\begin{equation}\label{eq:vanishing_remainder_evolution}
 V_{\alpha}(\tau_n^j)W_n^j(\cdot+x_n^j)\rightharpoonup 0   
\end{equation}
as $n \to +\infty$. Note that, for $l<j$, 
\begin{equation}
    V_{\alpha}(\tau_n^l) W_n^{j-1}(\cdot + x_n^l) 
    = V_{\alpha}(\tau_n^l) W_n^{l-1}(\cdot + x_n^l)  - \phi^l  - \sum_{k=l+1}^{j-1}
    V_{\alpha}(\tau_n^l - \tau_n^k) \phi^{k}(\cdot + x_n^l - x_n^k),
\end{equation}
therefore, if \eqref{eq:orthog_params} holds for all $k, l < j$, we have that $ V_{\alpha}(\tau_n^l - \tau_n^k) \phi^{k}(\cdot + x_n^l - x_n^k) \rightharpoonup 0$ 
for all $k, l < j$, and since $ V_{\alpha}(\tau_n^l) W_n^{l-1}(\cdot + x_n^l)  - \phi^l \rightharpoonup 0$, we have  $V_{\alpha}(\tau_n^l) W_n^{j-1}(\cdot + x_n^l)\rightharpoonup 0$. Since \eqref{eq:weak_lim_j}  holds, we then have $|\tau_n^j - \tau_n^l|+|x_n^j-x_n^l| \to +\infty$ (otherwise $\phi^j = 0$). By induction, property \eqref{eq:orthog_params} is then proved.

For the Pythagorean expansions \eqref{eq:L2_expansion} and \eqref{eq:Hs_expansion}, note that, for $s \in \{0,\frac\alpha2\}$:
\begin{align}
    \|W_n^{j-1}\|_{\dot H^{s_c}}^2 
    &= \|V_{\alpha}(-\tau_n^j)\phi^j(\cdot - x_n^j)\|_{\dot H^{s}}^2 + \|W_n^j\|_{\dot H^{s}}^2 + 2\Re\inn{\phi^j}{V_{\alpha}(\tau_n^j)W_n^j(\cdot + x_n^j)}_{\dot H^{s}} \nonumber \\
    &= \|\phi^j\|_{\dot H^{s}}^2 + \|W_n^j\|_{\dot H^{s}}^2 + o_n(1).\label{eq:energy_step}
\end{align}
Iterating \eqref{eq:energy_step} from $j=1$ to $J$ and recalling $R_n^0 = u_n$, we obtain
\begin{equation}
    \|u_n\|_{\dot H^{s}}^2 = \sum_{j=1}^J \|\phi^j\|_{\dot H^{s}}^2 + \|W_n^J\|_{\dot H^{s}}^2 + o_n(1),
\end{equation}
which imply \eqref{eq:L2_expansion} and 
\eqref{eq:Hs_expansion}. The energy expansion \eqref{eq:energy_expansion} then follows from \eqref{eq:L2_expansion}, \eqref{eq:Hs_expansion}, continuity in time of the linear operator if $\tau_n^j$ is bounded, and vanishing of the $L^{k+2}(\R)$-norm of the linear evolution (and a standard density argument) if $\tau_n^j$ is unbounded.
For the remainder, we have, by Lemma \ref{lem:inverse_strichartz},
\begin{equation}
    \sum_{j=1}^{\infty} \delta_j^{\frac{2\alpha (2\alpha+1)\eta}{3(\alpha-1)(\alpha-2s_c)(\eta-1)}} \lesssim \sum_{j=1}^{\infty} \|\phi_j\|_{H^{\frac\alpha2}}^2 \leq \|v_n\|_{H^{\frac\alpha2}}^2 + o_n(1).
\end{equation}
Since $(v_n)_n$ is bounded in $H^{\frac\alpha2}(\R)$, we conclude \eqref{eq:smallness}. Finally, property \eqref{eq:weak_conv} is obtained by considering \eqref{eq:decomp} at rank $k-1$, then applying \eqref{eq:vanishing_remainder_evolution} and using \eqref{eq:orthog_params} for the remaining sum. This completes the proof of Theorem \ref{thm:main}.
\end{proof}

\subsection{Long-time perturbation}

We now focus on the long-time perturbation theory, which is a fundamental stability tool in the scattering analysis. It ensures that approximate solutions with small error and controlled spacetime norm remain close to genuine solutions.

\begin{prop}[Perturbation theory for \eqref{DGBO}] \label{prop:perturbation}
Let $\alpha \in (1,2]$, $k \geq 4$ be even, and $s_c = \frac12 - \frac{\alpha}{k}$. 
Let $u_0 \in H^{\frac\alpha2}(\R)$, and $\tilde u$ be a global solution (in the sense of Definition \ref{def gwp}) to the perturbed equation
\begin{equation} \label{perturbed_eq}
\partial_t \tilde u + D^\alpha  \partial_x \tilde u + \partial_x(\tilde{u}^{k+1}) = \partial_x e
\end{equation}
satisfying
\begin{equation}
\|u_0 - \tilde u(0)\|_{H^{\frac{\alpha}{2}}} +\|\tilde u\|_{X(\R_+)} +\|\tilde u\|_{S(\R_+)} + \|e\|_{N^{\frac{\alpha+1}{2}}(\R_+)}\leq L<+\infty.
\end{equation}

There exists $\epsilon_0(L)>0$ such that, for any $0<\epsilon<\epsilon_0$, if
\begin{equation}
	\|V_\alpha(t)(u_0 - \tilde u(0))\|_{S(\R_+)} +\|e\|_{\dot N^{s_c}(\R_+)} \leq \epsilon,
\end{equation}
then the corresponding global solution $u$ to the defocusing $k$-dgBO equation \eqref{DGBO} with initial data $u_0$ satisfies
\begin{equation}\label{perturb1}
\|u - \tilde u\|_{S(\R_+)} \lesssim_L \epsilon,  \quad\text{and}\quad \|u - \tilde u\|_{X(\R_+)} \lesssim_L 1,
\end{equation}
\end{prop}
\begin{proof}
	Let $0<\delta_0^{k-1} \ll (1+L)^{-1}$ and split $[0,+\infty) = \displaystyle\bigcup_{j=1}^{N(L)} I_k$ into a finite union of intervals $I_j = [t_j, t_{j+1})$ such that $\|\tilde u\|_{S(I_j)}<\delta_0$ for all $j$. 
	Let $w = u- \tilde u $. On each interval $I_j$, $w$ solves
	\begin{equation}
		w(t) = V_{\alpha}(t-t_j)w(t_j)+\int_{t_j}^{t}V_{\alpha}(t-\tau) \partial_x( (\tilde u + w)^{k+1} - \tilde u^{k+1}+e)(\tau) \, d \tau.
	\end{equation}
	Therefore, there exists $C = C(\alpha,k) > 2
    $ such that, for $t \in I_j$,
	\begin{align}
		\|w\|_{X([t_j,t])} 
		& \leq C\|w(t_j)\|_{H^{\frac{\alpha}{2}}} + C\left(\delta_0 + \|w\|_{S(I_j)}\right)^k\|w\|_{X(I_j)} + CL\left(\delta_0 + \|w\|_{S(I_j)}\right)^{k-1}\|w\|_{S(I_j)} + CL,
	\end{align}
	and
	\begin{align}
		\|w\|_{S([t_j,t])} 
		& \leq \|V_{\alpha}(t-t_j)w(t_j)\|_{S(I_j)} + C\left(\delta_0 + \|w\|_{S(I_j)}\right)^k\|w\|_{S(I_j)} + CL\left(\delta_0 + \|w\|_{S(I_j)}\right)^{k-1}\|w\|_{S(I_j)} + C\epsilon.
	\end{align}
	Upon iterating from $I_0$ to $I_j$ and bootstrapping, we have
	\begin{equation}
		\|w\|_{S(I_j)} \leq 2^{N(L)}C\epsilon, \quad \|w\|_{S(I_j)} \leq (2C)^{N(L)} L
	\end{equation}
	for any $j$, as long as we impose $ \,(2C)^{N(L)}\epsilon_0^{k-1} \ll (1+L)^{-1}$.
\end{proof}

\subsection{Compactness}
We now show how minimality allows us to recover compactness.
Let 
\begin{equation}
    L(A) = \sup\left\{ \|u\|_{S(\R)}:  M[u_0]+ E[u_0] < A\right\}
\end{equation}
and
\begin{equation}
    A_c = \sup \{ A \geq 0 : L(A)<+\infty\}.
\end{equation}
The small data theory guarantees $A_c>0$. The perturbation theory, in turn, states that $L: [0,+\infty) \to [0,+\infty]$ is continuous, since, for $f \neq 0$,
\begin{equation}
    \frac{d}{d\eta}\left(M[(1+\eta)f]+E[(1+\eta)f]\right)\Big|_{\eta=0} = 2\|f\|_{L^2}^2+\|D^{\frac{\alpha}{2}}f\|_{L^2}^2+\|f\|_{L^{k+1}}^{k+1} > 0.
\end{equation}

If the conclusion of Theorem \ref{teo:scat} fails, then $A_c<+\infty$. In this case, we show that $L(A_c) = +\infty$. Moreover, we also show that a special solution can be constructed exactly at the level $A_c$.

\begin{lem}\label{lem:one_profile}
    Suppose $A_c<+\infty$. Let $(u_n)_n$ be a sequence of solutions with initial data $(u_{n,0})_n \subset H^{\frac\alpha2}(\R)$ and $(t_n)_n$ be a sequence of time translations such that $M[u_n]+E[u_n] \to A_c$ for all $n$ and $\|u_n\|_{S([t_n,+\infty))} \to +\infty$. Then, there exist a sequence of space translations $(x_n)_{n}$ and $v_0 \in H^{\frac\alpha2}(\R)$ such that $u_{n}(t_n,\cdot+x_n) \to v_0$ in $H^{\frac\alpha 2}(\R)$. Moreover, the corresponding solution $v$ to \eqref{DGBO} with initial data $v_0$ satisfies $M[v]+E[v]=A_c$ and $\|v\|_{S(\R_+)} = \infty$.
\end{lem}
\begin{proof}
Denote by $DgBO(t) v$ the solution of \eqref{DGBO} with initial data $v$. Passing to a subsequence, we write the linear profile decomposition (Theorem \ref{thm:main})

\begin{equation}
    u_n(t_n) = \sum_{j=1}^{J} V_{\alpha}(\tau_n^j) \phi^j(x - x_n^j) + W_n^J(x).
\end{equation}
If $\tau_n^j$ has a bounded subsequence, we can assume $\tau_n^j \to \tau \in \R$. In this case, define $\tilde\phi^j(t) = DgBO(t-\tau)V_\alpha(\tau) \phi^j$. If not, we can assume $\tau_n^j \to +\infty$ or $\tau_n^j \to -\infty$ as $n \to \infty$ and define $\tilde \phi^j(t)$ as the wave operator given by Proposition \ref{prop:wave_operator} such that
\begin{equation}
    \|\tilde\phi^j(t) - V_{\alpha}(t)\phi^j\|_{H^{\frac\alpha 2}_x}\to 0, \quad\text{ as }\quad t \to \pm \infty.
\end{equation}
We then write
\begin{equation}
    u_n(t_n,x) = \sum_{j=1}^{J} \tilde\phi^j(\tau_n^j,x - x_n^j) + \tilde W^J_n,
\end{equation}
with
\begin{equation}
    \tilde W^J_n = \sum_{j=1}^{J} (V_{\alpha}(\tau_n^j) \phi^j(x - x_n^j)-\tilde \phi^j(\tau_n^j,x - x_n^j)) + W^J_n.
\end{equation}
Moreover, we define
\begin{equation}
    \tilde \phi^j_n(t,x) = \tilde\phi^j(t+\tau_n^j,x-x_n^j)
\end{equation}
and
\begin{equation}
    \tilde u^J_n = \sum_{j=1}^J \tilde \phi^j_n.
\end{equation}

There are two cases to consider:

\underline{Case 1} (no bad profiles): There exists $\eta>0$ such that $\displaystyle\sup_j\left(M[\phi^j] + \displaystyle\limsup_{n \to +\infty} E[V_{\alpha}(\tau^j_n)\phi^j]\right) \leq A_c-\eta$.

In this case, $ M[\tilde \phi^j] +  E[\tilde \phi^j] \leq A_c - \eta$ for all $j \geq 0$, which implies, by Theorem \ref{prop:scattering_criterion}, $\|\tilde\phi^j\|_{X(\R) \cap S(\R)} \lesssim_{\eta}1$ for all $j$.  Moreover, there exists $J_0>0$ such that 
\begin{equation}
    \sum_{j= J_0}^{\infty} \|\phi^j\|^2_{H^{\frac \alpha 2}}\ll\delta_{sd}^2.
\end{equation}
Therefore, for any $J \geq J_0$, 
\begin{equation}
    \lim_{n \to +\infty }\sum_{j=J_0}^J\|\tilde \phi^j(\tau_n^j)\|_{H^{\frac{\alpha}{2}}}^2 \lesssim \sum_{j=J_0}^J\|\phi^j\|_{H^{\frac{\alpha}{2}}}^2 \ll \delta_{sd}^2.
\end{equation}
This implies, by small data theory, that 
\begin{equation}
    \sum_{j=J_0}^J\|\tilde \phi^j\|_{X(\R) \cap S(\R)}^2 \lesssim \delta_{sd}^2.
\end{equation}

Thus, for any $J \geq 1$, by the elementary inequality
\begin{equation}
\label{eq:elementary_inequality}
\left|\Bigg| \sum_{j=1}^J z_j \Bigg|^a - \sum_{j=1}^J |z_j|^a \right|\leq C_J \sum_{\substack{l \neq m\\l,m\leq J}} |z_l| |z_m|^{a-1}, \quad a > 1,
\end{equation}
we have, since $p_{s_c}\geq 2$ and $2(\alpha-1)/(\alpha+2)<1$,
\begin{align}
    \|\tilde u^J_n\|_{L^{p_{s_c}}_xL^{q_{s_c}}_t}^{p_{c}} &= \int \left(\int |\tilde \phi^j_n|^{q_{c}}\, dt \right)^{\frac{2(\alpha-1)}{\alpha+2}}\, dx\\
    &\leq \sum_{j=1}^{J_0-1} \|\tilde \phi^j\|_{S(\R)}^{p_{s_c}} +  \left(\sum_{J_0\leq j \leq J} \|\tilde \phi^j\|_{S(\R)}^{2} \right)^{\frac{p_{s_c}}{2}}  +  C_J\left(\sup_{j\geq 1}\|\tilde \phi^j\|_{S(\R)}^{\frac{(2\alpha+1)k-4(\alpha-1)}{\alpha+2}}\right)\sum_{\substack{l \neq m\\l,m\leq J}}\|\tilde\phi^l_n\tilde \phi^m_n\|_{L^{\frac {p_{s_c}} {2}}_xL^{\frac {q_{s_c}} {2}}_t}
    \end{align}
    
which implies, due to $q_{s_c}>1$ and the asymptotic orthogonality \eqref{eq:orthog_params}, 
\begin{equation}
\sup_{J}\left(\limsup_{n \to \infty}\|\tilde u^J_n\|_{L^{p_{s_c}}_xL^{q_{s_c}}_t}\right) \lesssim_{J_0} \sum_{j=1}^{J_0} \|\tilde \phi^j\|_{S(\R)} + \delta_{sd}
< +\infty.
\end{equation}

Similarly, $\displaystyle\sup_{J}\left(\limsup_{n \to \infty}\|D^{s_c+\frac12}\tilde u^J_n\|_{L^{p_0}_xL^{q_0}_t}\right)+\displaystyle\sup_{J}\left(\limsup_{n \to \infty}\|\tilde u^J_n\|_{X(\R)}\right) \lesssim_{J_0}1< +\infty$. We now observe that, by the Pythagorean expansions \eqref{eq:L2_expansion} and \eqref{eq:Hs_expansion}
\begin{equation}
    \limsup_{n \to +\infty}\|u_n(t_n) - \tilde u_n^j(0)\|_{H^{\frac{\alpha}{2}}} \leq \limsup_{n \to +\infty}\|\tilde W_n^{J}\|_{H^{\frac \alpha 2}} \leq \|u_n(t_n)\|_{H^{\frac \alpha 2}} \lesssim_{A_c} 1.
\end{equation}
Moreover, this remainder vanishes in the Strichartz norm:
\begin{equation}
    \limsup_{n \to +\infty}\|V_{\alpha}(t)(u_n(t_n) - \tilde u_n^j(0))\|_{S(\R)}   \leq \limsup_{n \to +\infty}\|V_{\alpha}(t) W_n^J\|_{H^{\frac \alpha 2}} = o_J(1).
\end{equation}
Finally, the error in the approximation can be directly estimated by
\begin{align}
    \left\|\left(\sum_{j=1}^J \tilde \phi^j_n\right)^{k+1}-\sum_{j=1}^J (\tilde \phi^j_n)^{k+1}\right\|_{N^{\frac{\alpha}{2}}(\R)} &\leq  \sup_{j\geq 1} \|\tilde \phi^j\|_{X(\R)\cap S(\R)}^{k+1} \lesssim_{\eta} 1
\end{align}
and
\begin{align}\label{eq:err_sk_1}
    \left\|\left(\sum_{j=1}^J \tilde \phi^j_n\right)^{k+1}-\sum_{j=1}^J (\tilde \phi^j_n)^{k+1}\right\|_{\dot N^{s_c}(\R)} 
   &\leq \left(\sup_{j\geq 1} \|\tilde \phi^j\|_{ S(\R)}^{k-1} \right)\sum_{\substack{l \neq m\\l,m \leq J}} \|\tilde \phi^l_n \tilde \phi^m_n\|_{L^ {\frac {p_{s_c}} {2}}  _xL^{\frac {q_{s_c}} {2}}_t(\R \times \R)}\\
   &\quad+  \left(\sup_{j\geq 1} \|\tilde \phi^j\|_{S(\R)}^{k-1} \right)\sum_{\substack{l \neq m\\l,m \leq J}} \|\tilde \phi^l_n \tilde \phi^m_n\|_{\dot H^{s_c+\frac{1}{2},p_1}_xL^{q_1}_t(\R \times \R)},\label{eq:err_sk_2}
\end{align}
with $p_1 = \frac{(4\alpha+2)k}{k(\alpha-1)+2(\alpha+2)}$, $q_1 = \frac{(4\alpha+2)k}{4(\alpha-1)+3k}$. The orthogonality condition \eqref{eq:orthog_params} implies
\begin{equation}
    \limsup_{n\to+\infty} \eqref{eq:err_sk_1} = o_J(1).
\end{equation}
To deal with \eqref{eq:err_sk_2}, we observe that $s_c + \frac{1}{2}<1<\frac{\alpha+1}{2}$ and interpolate
\begin{align}
    \|\tilde \phi^l_n \tilde \phi^m_n\|_{\dot H^{s_c+\frac{1}{2},p_1}_xL^{q_1}_t(\R \times \R)} &\lesssim \|\tilde \phi^l_n \tilde \phi^m_n\|_{L^{p_1}_xL^{q_1}_t(\R \times \R)} + \|\partial_x(\tilde \phi^l_n \tilde \phi^m_n)\|_{L^{p_1}_xL^{q_1}_t(\R \times \R)}\\
    &\lesssim \|\tilde \phi^l_n \tilde \phi^m_n\|_{L^{p_1}_xL^{q_1}_t(\R \times \R)} + \|\partial_x(\tilde \phi^l_n) \,\tilde \phi^m_n\|_{L^{p_1}_xL^{q_1}_t(\R \times \R)}+\|\tilde \phi^l_n \,\partial_x(\tilde \phi^m_n)\|_{L^{p_1}_xL^{q_1}_t(\R \times \R)},
\end{align}
to obtain, again by \eqref{eq:orthog_params},
\begin{equation}
    \limsup_{n \to +\infty}\eqref{eq:err_sk_2} = o_J(1).
\end{equation}
We then obtain $J_1>0$ and, for any $J \geq J_1$, $n_1(J)>0$ such that, for any $n \geq n_1(J)$, invoking Proposition \ref{prop:perturbation}
\begin{equation}
    \|u_n\|_{S(\R)} \lesssim \|\tilde u^J_n\|_{S(\R)} \lesssim_{J_0} 1,
\end{equation}
which is a contradiction. Therefore, Case 1 is precluded.

\underline{Case 2} (exactly one bad profile): $\displaystyle\sup_j\left(M[\phi^j] + \displaystyle\limsup_{n \to +\infty} E[V_{\alpha}(\tau^j_n)\phi^j]\right) = A_c$.

In this case, by the Pythagorean expansions \eqref{eq:L2_expansion} and \eqref{eq:Hs_expansion}, we must have exactly one $j \geq 1$ such that $\phi^j \neq 0$. Therefore, we can write
\begin{equation}
    u_n(t_n,x) =  V_{\alpha}(\tau_n) \phi(x - x_n) + W_n(x),
\end{equation}
with
\begin{equation}
    M[\phi]+\limsup_{n\to+\infty}E[V_{\alpha}(\tau_n) \phi]=A_c,\quad\text{and}\quad\limsup_{n\to+\infty}\left(M[W_n]+E[W_n]\right) = 0.
\end{equation}
We then claim that $\tau_n$ has a convergent subsequence. If not, we can assume either $\tau_n \to +\infty$ or $\tau_n \to -\infty$. To preclude the former, we compute
\begin{align}
    \|V_\alpha(t)u_n(t_n)\|_{S(\R_+)} &\leq \|V_\alpha(t+\tau_n)\phi\|_{S(\R_+)}+\|V_\alpha(t)W_n\|_{S(\R_+)}\\
    &=\|V_\alpha(t)\phi\|_{S([\tau_n,+\infty))}  +\|V_\alpha(t)W_n\|_{S(\R)}\\
    &=o_n(1),
\end{align}
which would imply, by small data theory, $\|u_n\|_{S(\R_+)}<+\infty$ for $n$ sufficiently large. For the former, we compute
\begin{align}
    \|V_\alpha(t)u_n(t_n)\|_{S(-\infty,-\frac{\tau_n}{2})} &\leq \|V_\alpha(t+\tau_n)\phi\|_{S(-\infty,-\frac{\tau_n}{2})}+\|V_\alpha(t)W_n\|_{S(\R)}\\
    &=\|V_\alpha(t)\phi\|_{S(-\infty,\frac{\tau_n} 2)}  +\|V_\alpha(t)W_n\|_{S(\R)}\\
    &=o_n(1),
\end{align}
which then would imply $\|u_n\|_{S(\R)}<+\infty$. Therefore, the only possibility is to have, passing to a subsequence and using time-translation (possibly redefining $\phi$), $\tau_n \to 0$. Define $v_0 = \phi$, $v(t) = DgBO(t)v_0$ and $\tilde \phi (t) = v(t,\cdot-x_n)$. We then have $u_n(t_n,\cdot+x_n) \to v_0$ in $H^\frac{\alpha}{2}$ and $M[v]+E[v]=M[v_0]+E[v_0]=A_c$. Moreover, since $\|u_n(t_n)-\tilde \phi(0)\|_{X(\R)} \to 0$ as $n \to +\infty$, we must have $\|v\|_{S(\R_+)}=\|\tilde \phi\|_{S(\R_+)}=+\infty$, by perturbation theory. This concludes the proof.
\end{proof}

\subsection{The critical solution}

We now have all the tools needed for the construction:

\begin{proof}[Proof of Proposition \ref{prop:nontrivial_localization}]

If $A_c<+\infty$, then there exists a sequence $u_n$ of solutions such that $M[u_n]+E[u_n] \nearrow A_c$ and $\|u_n\|_{S(\R)} \nearrow +\infty$. By time translation, we can assume $\|u_n\|_{S(\R_+)} \nearrow +\infty$.

Applying Lemma \ref{lem:one_profile} to $u_n$ with $t_n \equiv 0$ guarantees $u_{c,0} \in H^{\frac{\alpha}{2}}(\R)$ such that the corresponding solution $u_c$ satisfies $M[u_c]+E[u_c]=A_c$ and $\|u_c\|_{S(\R_+)} = \infty$. Now suppose \eqref{eq:nontrivial_localization} does not hold. Then we have a sequence of times $(t_n)_n$ and a sequence of radius $R_n\to +\infty$ such that, for any choice of $x_n \in \R$, 
\begin{equation}\label{eq:vanishing_uc}
    \int_{|x|\leq R_n} |u_c(t_n,x+x_n)|^{k+2}\, dx \to 0.
\end{equation}
Note that, since $u_c$ blows up forward in time, $\|u_c\|_{S([t_n,+\infty))} =+\infty$ for all $n$. Therefore we are able to apply Lemma \ref{lem:one_profile} again to assert the existence of $(x_n)_n$ such that, after passing to a subsequence, $u_c(t_n, \cdot + x_n)\to v_0 \in H^{\frac{\alpha}{2}}(\R)\backslash \{0\}$. This implies the existence of $\tilde R(v_0)>0$ such that
\begin{equation}
       \int_{|x|\leq \tilde R(v_0)} |u_c(t_n,x+x_n)|^{k+2}\, dx \to \int_{|x|\leq \tilde R(v_0)} |v_0(x)|^{k+2}\, dx >0,
\end{equation}
which contradicts \eqref{eq:vanishing_uc}

\end{proof}

\section{The fractional Laplacian}\label{sec-4}

In this section, we recall several fundamental tools related to the fractional
Laplacian. We begin with the Caffarelli-Silvestre extension, which provides a
useful local representation of $D^\beta$ for $\beta \in (1,2)$. We then prove
a number of functional identities and formulas that will be employed
throughout our work.

\subsection{The Caffarelli-Silvestre Extension }
In this subsection we recall the extension characterization of the fractional Laplacian, following the celebrated approach of Caffarelli and Silvestre \cite{CS07}. We work here in the one-dimensional setting, although all arguments extend naturally to $\R^n$.


\subsubsection{The extension problem}

Given $u:\R\to\R$, we consider its \emph{$\beta$-fractional extension} $U:\R\times[0,\infty)\to\R$ defined by the elliptic equation
\begin{equation}\label{eq_ext_u}
    \begin{cases}
    \partial_{xx}^2U(x,z) + \displaystyle\frac{1-\beta}{z}\partial_z U(x,z) + \partial_{zz}^2 U(x,z)=0, &x \in \R, \, z> 0\\
    U(x,0) = u(x), & x \in \R.
    \end{cases}
\end{equation}
In divergence form, \eqref{eq_ext_u} becomes
\begin{equation}\label{eq:div_form}
\operatorname{div}\!\left( z^{1-\beta}\nabla U\right)=0,
\end{equation}
which is the Euler-Lagrange equation for the weighted Dirichlet functional
$$
\mathcal{E}(U):=\iint_{\R\times\R_+} z^{1-\beta} |\nabla U|^2 \, dx\,dz.
$$
The weight $z^{1-\beta}$ belongs to the Muckenhoupt $A_2$ class for $\beta\in(0,2)$ (see \cite{Grafakos}), so solutions of
\eqref{eq:div_form} share many of the regularity properties of classical harmonic functions. The Poisson kernel associated with the equation \eqref{eq:div_form}
is explicitly given by 
\begin{equation}\label{eq:poisson_kernel_cs}
P_z(x):=P(x,z)
= C_\beta\frac{z^\beta}{(x^2+z^2)^{\frac{\beta+1}{2}}},
\end{equation}
where $C_\beta :=\frac{\Gamma\!\left({\beta}/{2}+1/2\right)}{\pi^{1/2} \, \Gamma\!\left({\beta/2}\right)} $ is a normalization constant. This kernel satisfies the homogeneity relation
$$
P_{(\lambda z)}(\lambda x) = |\lambda|^{-1} P_{z}( x),
$$
which is equivalent to the scaling identity
$$
P_z(x)=z^{-1}P_{1}\left(\frac{x}{z}\right)
$$
In particular, for every $z>0$ the kernel has unit mass,
\begin{equation}\label{eq:poisson_mass_ks}
\int_{\R} P_z(x-y)\, dy= 1,
\end{equation}
and $P_z \rightharpoonup \delta_0$ in the distributional sense as $z \to 0^+$.  
This homogeneity structure is fundamental for the representation
\begin{equation}\label{Poisson Kernel U}
    U(x,z) =C_\beta\int \frac{z^\beta}{((x-w)^2 + z^2)^{\frac{\beta+1}{2}}}f(w)\, dw := (P_z *u)(x),
\end{equation}

This Poisson kernel also provides a maximum principle (see \cite{steinbook}, Page 62, Theorem 2.(a))
\begin{equation}
    \sup_{z>0}|U(x,z)| \leq \M u(x),
\end{equation}
where $\M$ is the Hardy-Littlewood maximal function defined in \ref{fun. Maximal}. Moreover, the following energy identity, established in \cite[Equation (3.7)]{CS07}, holds:
\begin{equation}\label{euqacao 3.7 caffarelli}
\iint_{\mathbb{R}\times\mathbb{R}_+} z^{1-\beta}\,|\nabla U(x,z)|^2 \,dx\,dz
= \bigl\| D^{\frac{\beta}{2}} u \bigr\|_{L^2(\mathbb{R})}^2.
\end{equation}
 If, in particular, $u \in L^2(\R)$, one can represent the solution, using the Fourier transform in the $x$ variable, as
\begin{equation}\label{fourier_fractional}
    \hat U(\xi,z) = G(|\xi|z)\hat u(\xi),
\end{equation}
where $G$ satisfies the following second-order IVP
\begin{equation}
    \begin{cases}
         &G''(z)+ \frac{1-\beta}{z}G'(z)- G(z) = 0,\\
         \hfill &G(0) = 1,\\
    \end{cases}
\end{equation}
and satisfies $\displaystyle \lim_{z \to +\infty}G(z) = 0$. By a direct computation, one can write $G$ as
\begin{equation}
G(z) = \dfrac{2^{1-\beta/2}\,z^{\beta/2}}{\Gamma\!\left({\beta}/{2}\right)}\,K_{\beta/2}(z),
\end{equation}
where $K_{\nu}$ is the modified Bessel function (see \cite[p. 375, Equation 9.6.2]{Bessel_reference}). Note that the equation \eqref{fourier_fractional} permits us relate $\nabla_{x,z}U$ and $\partial_x u$. Indeed, we have
\begin{align}
    \partial_x U(x,z) &= P_z * \partial_x u, \text{ and }\\
    \partial_z U(x,z) &= \left[\tfrac{x}{z}P_z \right]* \partial_x u,
\end{align}
where we note that the function $x \mapsto \tfrac{x}{z}P_z(x)$ belongs to $L^1(\R)$ if $\beta>1$. These identities allow us to establish further useful relations (see, for example, the proof of Lemma \ref{lem:claim_tilde_Phi}).

The main result of the extension method is that the fractional Laplacian appears as a boundary flux for the degenerate elliptic equation \eqref{eq_ext_u}.
\begin{prop}[Caffarelli-Silvestre, \cite{CS07}]\label{limit_neumann}
Let $u:\R\to\R$ and let $U$ solve the extension problem \eqref{eq_ext_u}. Then
\begin{equation}\label{eq:cs_limit}
D^\beta u(x)
= -c \lim_{z\to0^+} z^{1-\beta} \partial_z U(x,z),
\end{equation}
where $c = c(\beta) > 0$ is a constant depending only on $\beta$. Moreover, $U$ and $z^{1-\beta}\partial_z U$ are continuous on the half-space $\mathbb{R}\times\overline{\mathbb{R}_+} $.
\end{prop}

The limit in \eqref{eq:cs_limit} exists whenever $D^\beta u$ is well defined, and one recovers the classical singular integral formula
\begin{equation}\label{eq:gagliardo_singular_integral}
D^\beta u(x)
= c\,  \,\textbf{p.v}\int_{\R} \frac{u(x)-u(z)}{|x-z|^{1+\beta}}\, dz .    
\end{equation}

The characterization in \eqref{eq:cs_limit} will play a fundamental role throughout this work.

\subsection{Some functional identities}

To begin, let $\phi$ be a smooth bump function which equals $1$ on the interval $[-1,1]$ and is supported in $[-2,2]$ and, for $\beta \in (1,2]$, denote by $\Phi = \Phi^\beta(x,z)$ the $\beta$-fractional Caffarelli-Silvestre extension of $\phi$.

\begin{lem}\label{lem: Identidade 1}
Let $\beta \in (1,2]$, $u \in H^{2}(\R)$ and $U = U^\beta(x,z)$ be the $\beta$-fractional extension of $u$.  
Then there exist constants $c_1, c_2 > 0$, depending only on $\beta$, such that
\begin{equation}\label{eq_Phi_uk}
    \int_{\R} \phi\, u^{\,k+1} D^{\beta}u \,dx
    \;=\;
    c_1 \iint_{\R\times\R_{+}} 
        \Phi \, z^{1-\beta} U^{k} |\nabla U|^{2}\,dxdz
    \;+\;
    c_2 \int_{\R} u^{\,k+2} D^{\beta}\phi\,dx .
\end{equation}
\end{lem}
Before proving Lemma \ref{lem: Identidade 1}, we briefly comment on its consequences. The following corollary is immediate.

\begin{cor}[Almost non-negativity] Let $\beta \in (1,2]$ and $u \in H^2(\R)$. Then, we have
    \begin{equation}\label{almost_non_negativity} 
    \int \phi\, u^{k+1} D^\beta u\,dx \geq-  c_2 \|u\|^{k+2}_{H^{\frac{\beta}{2}}}\|D^{\beta}\phi\|_{L^\infty}.
\end{equation}
\end{cor}

One can also make $\phi(x) \to 1$ in \eqref{eq_Phi_uk} to obtain:

\begin{cor}[Positivity Lemma]\label{positivity_lemma} Let $\beta \in (1,2]$, $u \in H^2(\R)\backslash\{0\}$, and $ k$ an even integer. Then the following inequality holds.
\begin{equation}\label{positivity_nophi}
    \int u^{k+1} D^\beta u \,dx> 0.
\end{equation}
\end{cor}
\begin{proof}[Proof of Corollary \ref{positivity_lemma}]


For $\lambda\geq 1$, define the rescaled cutoff
$$
\phi_\lambda(x) := \phi\!\left(\frac{x}{\lambda}\right),
$$
and let $\Phi_\lambda$ denote the $\beta$-fractional extension of $\phi_\lambda$. Replacing $\phi$ by $\phi_\lambda$ in \eqref{eq_Phi_uk}, by dominated convergence and Sobolev inequality we get
\begin{equation}
    \lim_{\lambda \to +\infty}\int \phi_\lambda u^{k+1} D^\beta u\,dx = \int u^{k+1} D^\beta u\,dx.
\end{equation}
Moreover,  by Sobolev embedding,

\begin{equation}
    \left|\int u^{k+2}D^{\beta} \phi_\lambda\,dx\right| \lesssim \|u\|_{L^{k+2}}^{k+2} \|D^\beta \phi_\lambda\|_{L^\infty} = 
    \lambda^{-\beta}\|u\|_{L^{k+2}}^{k+2} \|D^\beta \phi\|_{L^\infty}
    \lesssim \lambda^{-\beta}\|u\|_{H^\frac{\beta}{2}} \|D^\beta \phi\|_{L^\infty}.
\end{equation}
Therefore, the last term in \eqref{eq_Phi_uk} vanishes as $\lambda \to \infty$. For the remaining term, note that
\begin{equation}
    \Phi_\lambda(x,z) = \Phi_1\left(\frac{x}{\lambda}, \frac{z}{\lambda}\right) \longrightarrow \phi(0) = 1,
\end{equation}
pointwise, as $\lambda \to +\infty$. This can be seen directly from the homogeneity of the Poisson kernel associated to the fractional extension. Now, in view of \eqref{euqacao 3.7 caffarelli}, and since the integrand is bounded by
\begin{equation}
    \left|\Phi_\lambda\,  z^{1-\beta} U^k |\nabla U|^2\right| \leq
    \|\phi\|_{L^\infty} \|u\|_{H^\frac \beta2}^k \, z^{1-\beta} |\nabla U|^2,
\end{equation}
we have, by dominated convergence, as $\lambda \to +\infty$,
\begin{equation}
 \lim_{\lambda \to +\infty} \iint_{\R\times\R_+} \Phi_\lambda\,  z^{1-\beta} U^k |\nabla U|^2\,dxdz = \iint_{\R\times\R_+}   z^{1-\beta} U^k |\nabla U|^2\,dxdz.
\end{equation}
Thus, if \eqref{eq_Phi_uk} holds, then, for $u \in H^2(\R)\backslash\{0\}$,
\begin{equation}
    \int u^{k+1} D^\beta u \,dx= c_1 \iint_{\R\times\R_+}  z^{1-\beta} U^k |\nabla U|^2\,dxdz > 0.
\end{equation}
\end{proof}
\begin{re}
The previous identity in fact also yields
\begin{equation}\label{eq:frac_chain_rule_like}
    \int u^{\,k+1} D^{\beta}u\,dx
    \;=\; \frac{2}{k+2}
    \int \bigl[D^{\frac{\beta}{2}}\!\bigl(u^{\,\frac{k}{2}+1}\bigr)\bigr]^{2}\,dx,
\end{equation}
which may be interpreted as an ``almost'' fractional chain rule identity.
\end{re}

\begin{re}
 By replacing $\medint\int u D^\beta v$ in Lemma \ref{lem: Identidade 1} and Corollaries \ref{almost_non_negativity} and \ref{positivity_lemma} by the duality pairing $\langle D^\beta u, v \rangle = \medint\int D^{\frac \beta2}u \, D^{\frac \beta2}v$ in $H^{\frac{\beta}{2}}(\R)$, a density argument extends the result to $u$ merely in $H^{\frac \beta2}(\R)$. Since our approach relies on using well-posedness and persistence of regularity to work with smooth solutions, we state our results in $H^2(\R)$ to avoid unnecessary technicalities.
\end{re}

With the observations above, we now turn to the proof of Lemma \ref{lem: Identidade 1}:

\begin{proof}[Proof of Lemma \ref{lem: Identidade 1}: ] We start by noting that, for $u \in H^2(\mathbb{R})$, the functions $\mathcal M(u)$, $\mathcal M(\partial_x u)$ are bounded, respectively, by $\|u\|_{L^\infty}$ and $\|\partial_x u\|_{L^\infty}$ (therefore, by $\|u\|_{H^2}$). Moreover, both functions vanish as $|x|\to+\infty$.

Now, let $L$, $\epsilon > 0$ and consider the rectangle $\Omega^{L,\epsilon} = [-L,L]\times[\epsilon,L]$. Its boundary is given by the intervals
$I^{L,\epsilon}_0 = [-L,L]\times\{\epsilon\}$, $I^{L,\epsilon}_1 = \{L\}\times[\epsilon,L]$, $I^{L,\epsilon}_2 = \{-L\}\times[\epsilon,L]$ and $I_3^L = [-L,L] \times \{L\}$. By the divergence theorem and \eqref{eq:div_form}, we have
\begin{align}
    -c\int_{I_0^{L,\epsilon}} \Phi z^{1-\beta} U^{k+1} \partial_y U 
    &= c\int_{I_0^{L,\epsilon}} \Phi z^{1-\beta} U^{k+1} (\nabla U \cdot n)\,dxdz \\
    &= c \iint_{\Omega^{L,\epsilon}} \text{div}(\Phi z^{1-\beta}  U^{k+1}\nabla U)\,dxdz - c \int_{I_1^{L,\epsilon}\cup I_2^{L,\epsilon}\cup I_3^{L}} \Phi z^{1-\beta} U^{k+1}(\nabla U \cdot n)\,dxdz\\
    &= \frac{c}{k+2} \iint_{\Omega^{L,\epsilon}} z^{1-\beta} \nabla \Phi \cdot \nabla (U^{k+2})\,dxdz + c(k+1) \iint_{\Omega^{L,\epsilon}}  \Phi z^{1-\beta}  U^{k}|\nabla U|^2\,dxdz \\&\quad - c \iint_{\Omega^{L,\epsilon}}  \Phi   U^{k+1}\underbrace{\text{div}(z^{1-\beta}\nabla U)}_{=0}- c \int_{I_1^{L,\epsilon}\cup I_2^{L,\epsilon}\cup I_3^{L}} \Phi z^{1-\beta} U^{k+1} (\nabla U \cdot n)\\
    &= \frac{c}{k+2} \iint_{\Omega^{L,\epsilon}} z^{1-\beta} \nabla \Phi \cdot \nabla (U^{k+2}) + c(k+1) \iint_{\Omega^{L,\epsilon}}  \Phi z^{1-\beta}  U^{k}|\nabla U|^2 \\&\quad - c \int_{I_1^{L,\epsilon}\cup I_2^{L,\epsilon}\cup I_3^{L}} \Phi z^{1-\beta} U^{k+1} (\nabla U \cdot n),
\end{align}
where $c$ is the same constant as in \eqref{limit_neumann}. Since
\begin{align}
     \int_{I_0^{L,\epsilon}} z^{1-\beta} U^{k+2}(\nabla\Phi\cdot n) 
    &=  \iint_{\Omega^{L,\epsilon}} \text{div}(   U^{k+2} z^{1-\beta} \nabla \Phi) -  \int_{I_1^{L,\epsilon}\cup I_2^{L,\epsilon}\cup I_3^{L}}  z^{1-\beta} U^{k+2}(\nabla \Phi \cdot n)\\
    &=  \iint_{\Omega^{L,\epsilon}} z^{1-\beta} \nabla(U^{k+2}) \cdot \nabla \Phi  +  \iint_{\Omega^{L,\epsilon}}U^{k+2} \underbrace{\text{div}(z^{1-\beta} \nabla \Phi)}_{=0}  \\&\quad-  \int_{I_1^{L,\epsilon}\cup I_2^{L,\epsilon}\cup I_3^{L}}  z^{1-\beta} U^{k+2}(\nabla \Phi \cdot n)\\
    &=  \iint_{\Omega^{L,\epsilon}} z^{1-\beta} \nabla(U^{k+2})\cdot \nabla \Phi-  \int_{I_1^{L,\epsilon}\cup I_2^{L,\epsilon}\cup I_3^{L}}  z^{1-\beta} U^{k+2}(\nabla \Phi \cdot n),
\end{align}
we have
\begin{align}
    -c\int_{I_0^{L,\epsilon}} \Phi z^{1-\beta} U^{k+1} \partial_y U 
    &= c(k+1) \iint_{\Omega^{L,\epsilon}}  \Phi z^{1-\beta}  U^{k}|\nabla U|^2 - \frac{c}{k+2}\int_{I_0^{L,\epsilon}} U^{k+2} z^{1-\beta} \partial_y\Phi \\
    &\quad +\frac{c}{k+2}\int_{I_1^{L,\epsilon}\cup I_2^{L,\epsilon}\cup I_3^{L}}  z^{1-\beta} U^{k+2}(\nabla \Phi \cdot n)\\
     &\quad- c \int_{I_1^{L,\epsilon}\cup I_2^{L,\epsilon}\cup I_3^{L}} \Phi z^{1-\beta} U^{k+1} (\nabla U \cdot n).
\end{align}
We now show that we can make $\epsilon \to 0^+$ and $L \to +\infty$. The limit of the double integral follows by dominated convergence. On $I_0^{L,\epsilon}$, the limit follows from the (uniform) continuity of 
$U$, $\Phi$, $z^{1-\beta}\partial_y U$ and $z^{1-\beta}\partial_y \Phi$ on compact sets of $\mathbb{R} \times \overline{\mathbb{R}_+}$ to make $\varepsilon \to 0^+$, followed by dominated convergence for $L \to +\infty$. On $I_1^{L,\epsilon}$, we make use of
\begin{equation}
    \sup_x|\Phi(x,z)|\leq  \sup_x \left|\int P(x-w,z) \phi(z) \, dw \right| \lesssim \sup_x P(x,z) \|\phi\|_{L^1} \lesssim \frac{1}{z}\|\phi\|_{L^1}
\end{equation}
and
\begin{equation}
    \sup_x|\Phi(x,z)|\leq  \sup_x \left|\int P(x-w,z) \phi(z) \, dw \right| \lesssim \|P\|_{L^1} \|\phi\|_{L^\infty} =\|\phi\|_{L^\infty}
\end{equation}
to write
\begin{align}
    \left|\int_\epsilon^L \Phi(L,z) \, z^{1-\beta} U^{k+1}(L,z) \partial_x U(L,z) \, dz\right| 
    &\leq \int_0^\infty|\Phi(L,z)| \, z^{1-\beta} |U^{k+1}(L,z)| |\partial_x U(L,z)| \, dz  \\
    &\quad \lesssim  [\M(u)(L)]^{k+1}\|\partial_x u\|_{L^\infty} \left(\|\phi\|_{L^\infty} \int_0^1 z^{1-\beta} dz+ \|\phi\|_{L^1}\int_1^\infty z^{-\beta} dz\right)\\
    &\quad \lesssim [\M (u)(L)]^{k+1}\|u\|_{H^2}(\|\phi\|_{L^\infty} + \|\phi\|_{L^1})   \to 0,
\end{align}
uniformly in $\varepsilon$, when $L \to \infty$. The remaining integrals on $I^{L,\epsilon}_1$ and $I^{L,\epsilon}_2$ are treated analogously. On $I^{L}_3$, write
\begin{equation}
    \left|\int_{-L}^L \Phi(x,L) \, L^{1-\beta} U^{k+1}(x,L) \partial_y U(x,L) \, dx\right| \lesssim \frac{1}{L^{\beta-1}} \|\phi\|_{L^1} \|u\|^{k+2}_{L^\infty} \to 0,
\end{equation}
as $L \to \infty$, since
for $z \geq 1$, 
\begin{equation}
    |\partial_y P(x,z)| \lesssim P(x,z).
\end{equation}
The remaining integral on $I_3^L$ converges even faster to zero since $\partial_z \Phi$ decays faster than $\Phi$ as $z \to +\infty$. We conclude, for $c_1 =(k+1)c$ and $c_2 = 1/(k+2)$, and using \eqref{eq:cs_limit},
\begin{equation}
     \int \phi u^{k+1} D^\beta u \,dx= c_1 \iint_{\R\times\R_+} \Phi   z^{1-\beta} U^k |\nabla U|^2\,dxdz + c_2 \int u^{k+2}D^{\beta} \phi\,dx.
\end{equation}
\end{proof}

\section{Monotonicity formula}\label{sec-5}

In this section, we develop a monotonicity formula for the equation \eqref{DGBO}. Our approach is inspired by the ideas introduced by Tao in \cite{MR2276483} for the generalized KdV equation (corresponding to the case $\alpha = 2$ in \eqref{DGBO}), which were subsequently adapted by Dodson \cite{MR3625190}, in the mass-critical  and by Farah, Linares, Pastor, and Visciglia \cite{MR3772197} in the mass-supercritical case, and by Kim and Kwon for the defocusing generalized Benjamin–Ono equation in \cite{MR4009456} (the case $\alpha = 1$ in \eqref{DGBO}).
\begin{de}[Mass density and mass current] 
Let $\alpha>0$, $k \in \mathbb{Z}^+$, and $v \in H^\alpha(\R)$. 
The {mass density} is defined by
$$
    \rho[v] := v^2
$$
and the corresponding {mass current} is given by
$$
    j[v] := (\alpha+1)\, v D_x^\alpha v \;+\; \frac{2(k+1)}{k+2}\, v^{k+2}.
$$
\end{de}

\begin{de}[Energy density and energy current] 
Let $\alpha>0$, $k \in \mathbb{Z}^+$, and $v\in H^{2\alpha}(\R)$. 
The {energy density} is defined by 
$$
    e[v] := \frac{1}{2}\, v D_x^\alpha v \;+\; \frac{1}{k+2}\, v^{k+2},
$$
and the corresponding {energy current} is given by
$$
    \kappa[v] := \frac{\alpha+1}{2}\, v D_x^{2\alpha} v 
    \;+\; \frac{\alpha+2}{2}\, v^{k+1} D_x^\alpha v
    \;+\; \frac{1}{2}\, v^{2k+2}.
$$
\end{de}

\begin{re}
If $u=u(t,x)$ is a sufficiently regular solution of \eqref{DGBO}, we adopt the notational convention
$$
    \rho(t,x) = \rho[u(t,x)],
$$
and similarly for $e(t,x)$, $j(t,x)$, and $\kappa(t,x)$.
\end{re}

\begin{prop}\label{proposicao da massa e energia}
Let $\alpha \in (1,2]$, $k \geq 4$, and let $u=u(t,x)$ be a sufficiently regular solution of \eqref{DGBO}. 
Then the following identities hold:
\begin{equation}
    \partial_t \int x\hspace{0.2mm} \rho(t,x)\, dx = \int j(t,x)\, dx,
\end{equation}
and
\begin{equation}
    \partial_t \int x\hspace{0.2mm} e(t,x)\, dx = \int \k(t,x)\, dx.
\end{equation}
\end{prop}

\begin{proof}
We first consider the energy density. An explicit computation gives


\begin{align}
\label{termoA}\partial_t\int x \hspace{0.2mm} e(t,x) \, dx = &-\frac{1}{2}\int x D_x^\alpha \partial_x uD^\alpha_x u \, dx -\frac{1}{2}\int x uD_x^{2\alpha} \partial_x u\, dx\\
 \label{termoB}&-\frac{1}{2}\int x \partial_x(u^{k+1}) D^\alpha_x u\, dx -\frac{1}{2}\int x uD^\alpha_x\partial_x(u^{k+1})\, dx -\int x u^{k+1}D_x^\alpha \partial_x u\, dx\\
\label{termoC}& -\int x u^{k+1}\partial_x(u^{k+1})\, dx.
\end{align}




We begin with \eqref{termoA}. First note that
$$
D^{\alpha}_x \partial_x u \, D^{\alpha}_x u = \frac{1}{2}\partial_x[(D^{\alpha}_xu)^2].
$$

Moreover, for $\beta >0$ and $v \in \mathcal{S}(\mathbb R)$, by computing Fourier transforms, one has the identity 
\begin{equation}\label{identidade}
[D^\beta_x; x]\partial_x v= \beta D^\beta_x v, 
\end{equation}
which implies
\begin{equation}\label{identidade_2}
    \int x v D^\beta_x \partial_x v\, dx = - \frac{(\beta+1)}{2}\int v D^\beta_x v\, dx .
\end{equation}

We then have
\begin{equation}
\eqref{termoA} = \frac{\alpha+1}{2} \int u D^{2\alpha}_x u \, dx.
\end{equation}

For \eqref{termoB}, observe that, integrating by parts 
\begin{align}
-\frac{1}{2}\int x \partial_x(u^{k+1}) D^\alpha_x u\, dx&=\frac{1}{2} \int u^{k+1} D^\alpha_x u\, dx +\frac{1}{2}\int x u^{k+1}D_x^\alpha \partial_x u\, dx,
\end{align}
and, integrating by parts and using the identity \eqref{identidade},
\begin{align}
-\frac{1}{2}\int x uD^\alpha_x\partial_x(u^{k+1})\, dx 
 &=   \frac{1}{2}\int u^{k+1} D^\alpha_x u\, dx + \frac{1}{2} \int u^{k+1} [D^\alpha_x;x] \partial_x u\, dx + \frac12 \int x u^{k+1}D_x^\alpha \partial_x u\, dx\\
 &= \frac{\alpha + 1}{2} \int u^{k+1} D^\alpha_x u\, dx +\frac12 \int x u^{k+1}D_x^\alpha \partial_x u\, dx.
\end{align}
Therefore
\begin{equation}
\eqref{termoB}= \frac{\alpha+2}{2} \int u^{k+1} D^\alpha_x u\, dx.
\end{equation}

To conclude, for \eqref{termoC}, since $u^{k+1} \partial_x (u^{k+1}) =\frac{1}{2}\partial_x[(u^{k+1})^2] $,
\begin{align}
\eqref{termoC}&= \frac{1}{2} \int u^{2k+2}\, dx.
\end{align}
Therefore, we obtain the relation
\begin{equation}
\partial_t\int x e(t,x) \, dx =\int \k(t,x)\, dx.
\end{equation}

For the mass density, observe that
\begin{align}
\partial_t \rho 
= -2 u D^\alpha_x \partial_x u - 2u\partial_x(u^{k+1}).
\end{align}
Therefore 
\begin{align}\label{termoI}
\int x \partial_t \rho(t,x) \, dx 
= -2\int x u D_x^\alpha \partial_xu \, dx
-2\int x u\partial_x(u^{k+1})  \, dx.
\end{align}
Then, by \eqref{identidade} and integration by parts,
\begin{equation}
    \partial_t\int x  \rho(t,x) \, dx = \int j(t,x)\, dx.
\end{equation}
\end{proof}
Observe that, by Proposition \ref{proposicao da massa e energia}
\begin{equation}\label{formula de monoticidade I}
    \partial_t \iint (y-x) \rho(t,x) e(t,y)\,dx\,dy = \left(\int \rho  \int \k-\int j \int e \right)(t).
\end{equation}

We now derive the monotonicity formula, which guarantees that the expression in \eqref{formula de monoticidade I} exhibits a suitable form of coercivity. More precisely, the monotonicity formula captures the nature of the propagation: it shows that the center of energy travels to the right strictly faster than the center of mass.

\begin{teo}[Monotonicity formula]\label{teo_monotonicity}
Let $v \in H^{2\alpha}(\R)$. Then we have
$$
\int \rho[v]\,dx  \int \k[v]\,dx - \int e[v]\,dx\int j[v]\,dx  \geq\;\; 
\frac{k^2}{2 (k+2)^2} \,  \, \ \left(\int {v}^{k+2} \, dx\right)^2.
$$
\end{teo}
\begin{proof} First, observe that 
\begin{align}
\eqref{formula de monoticidade I} & =\left(\int v^2\, dx\right) \left(\int \frac{\alpha+ 1}{2}  v D_x^{2\alpha} v
+ \frac{\alpha + 2}{2}   v^{k+1} D^\alpha_x v 
+  \frac{1}{2} v^{2k+2} \, dx\right)\\
& \quad\quad\quad\quad\quad\quad-\left(\int \frac{1}{2}vD^\alpha_x v +\frac{1}{k+2} v^{k+2} \, dx \right)\left(\int (\alpha+1)\,v D^\alpha_x v + \frac{2(k+1)}{k+2} v^{k+2}\, dx \right)\\
&= (M[v])^2\left[\left( \frac{\alpha+ 1}{2} a^2
+ \frac{\alpha + 2}{2}  abs
+  \frac{1}{2} b^2\right)- 
\left(\frac{1}{2}aq +\frac{1}{k+2} br \right)\left( (\alpha+1)aq + \frac{2(k+1)}{k+2} br \right)\right],
\end{align}
where, by using the $L^2$ norm and inner product, we defined
\[
a = \frac{\|D^\alpha v\|_{L^2}}{\|v\|_{L^2}},\quad b= \frac{\|v^{k+1}\|_{L^2}}{\|v\|_{L^2}},
\]
and
\[
q= \frac{\langle v,D^\alpha_x v\rangle}{ \|v\|_{L^2}\|D^\alpha_x v\|_{L^2}},\quad
%
r=   \frac{\langle v,v^{k+1}\rangle}{ \|v\|_{L^2}\|v^{k+1} \|_{L^2}},\quad
s=  \frac{\langle v^{k+1},D^\alpha_xv\rangle}{ \|v^{k+1}\|_{L^2}\|D^\alpha_x v\|_{L^2}}.
\]

Observe that, by Cauchy-Schwarz, 
$q,r,s \leq 1$. Moreover, $q,r \geq 0$ immediately, whilst we have $s\geq 0$ due to \eqref{positivity_nophi}.
We write
\begin{align}
\frac{\eqref{formula de monoticidade I}}{(M[v])^2}-
\frac{k^2}{2(k+2)^2} r^2 b^2
= 
\renewcommand{\arraystretch}{0.8}\begin{bmatrix} a & b \end{bmatrix}
\begin{bmatrix}
\tfrac{\alpha+1}{2}(1-q^2) 
& 
\tfrac{\alpha+2}{4}s - \tfrac{k+\alpha+2}{2(k+2)}qr
\\[10pt]
\tfrac{\alpha+2}{4}s - \tfrac{k+\alpha+2}{2(k+2)}qr 
& 
\tfrac{1}{2}(1-r^2)
\end{bmatrix}
\begin{bmatrix} a \\ b \end{bmatrix}
\end{align}

Hence, it suffices to show that the bilinear form defined above, restricted to vectors with nonnegative entries, is nonnegative.  Since $0\leq q, r \leq 1$, we only need to verify
\begin{align}
 \sqrt{(\alpha+1)(1-q^2)(1-r^2)} \geq \sqrt{\alpha+1} (qr-s)\geq \dfrac{k+\alpha+2}{k+2}qr -\dfrac{\alpha+2}{2}s.
\end{align}

To proceed, we recall that 
$$
\begin{bmatrix}
1 & q & r \\
q & 1 & s \\
r & s & 1
\end{bmatrix}
$$
is, up to normalization, the Gramian matrix of $v$, $D_x^{\frac\alpha2}v$, and $v^{k+1}$, hence, positive semidefinite. By Sylvester's criterion, 
%
$
1 - q^{2} - r^{2} - s^{2} + 2qrs \geq 0
$, which implies
$$
qr- s \leq \sqrt{(1-q^2)(1-r^2)}.
$$
Since $\alpha \in (1,2]$ and $k \geq \sqrt{3}-1$, we have
$$
\frac{k + \alpha+2}{\sqrt{\alpha+1}(k+2) } \leq 1 \quad \text{and} \quad \frac{\alpha+2}{2 \sqrt{\alpha+1}} \geq 1,
$$
which implies, since $q,r,s\geq 0$,
\begin{align}
    \left( \frac{k + \alpha+2}{k+2 }\right)qr - \left( \frac{\alpha+2}{2 }\right)s \leq \sqrt{\alpha+1} (qr- s),
\end{align}
and concludes the proof.
\end{proof}

\section{Rigidity}\label{sec-7}
Throughout this section, we let $u$ be a global solution to \eqref{DGBO} with initial data $u_0$ and define
\begin{equation}
    A := M[u_0]+E[u_0].
\end{equation}
Note that, by conservation of mass and energy, we have $\|u\|_{L^\infty_t H^{\frac{\alpha}{2}}_x} \lesssim_A 1$

Due to the persistence of regularity and to the well-posedness theory, we can rely on density arguments and assume that the solution $u$ is smooth as long as all the bounds depend only on $A$.

\subsection{Localized Interaction Functional}
Let $R>1$ and $\varphi_R$ be a bump function such that $0\leq \varphi_R \leq 1$, $\varphi_R(x) = 1$ if $|x| \leq R$ and $\varphi_R(x) = 0$ if $|x| \geq R+ R^{3/4}$. We can impose $\|\partial^{j}\varphi_R\|_{L^\infty} \lesssim_{j} R^{-3j/4}$ for every positive integer $j$, which implies
\begin{equation}
    \|D^{\beta}\varphi_R\|_{L^2}+\|\mathcal H D^{\beta}\varphi_R\|_{L^2} \lesssim_{\beta}{R^{-1/4}}
\end{equation}
for all $\beta \in [1,+\infty)$.

We also define 
\begin{equation}\label{eq:def_psi_R}
\psi_R(x) = \displaystyle\int_0^x \varphi_R^2 * \varphi_R^2 (\ell) \, \frac{d\ell}{R},    
\end{equation}
so that one can write
\begin{equation}\label{psi_prime}
    \psi'_R(y-x) = \int_{-\infty}^{+\infty} \varphi_R^2(x-\ell)\varphi^2_R(y-\ell) \frac{d\ell}{R}.
\end{equation}
For a solution $u$ to \eqref{DGBO}, define the interaction functional
\begin{equation}\label{Z_def}
    Z_R(t) := \iint \psi_R(y-x) \rho(t,x)\,e(t,y)\, dx dy,
\end{equation}
where $\rho(t,x) = \rho[u(t,x)]$ and $e(t,x)=e[u(t,x)]$ denote the mass and energy densities, respectively. 

To control $Z_R$, we need to circumvent the fact that $e$ might not be pointwise positive (in fact, it is not clear whether $\|e\|_{L^1}$ is even finite). Localization with smooth functions is then crucial here.
\begin{lem}\label{localization_lemma}
Let $0\leq \beta\leq 2$, $f \in H^{\beta}(\R)$ and $\phi \in C^1(\R)$ with
\begin{equation}
\|\phi\|_{W^{1,\infty}} := \|\phi\|_{L^\infty} + \|\partial_x \phi\|_{L^\infty} <+\infty.
\end{equation}
Then,
\begin{equation}
\left|\int \phi \,f\, D^\beta_x f \right|  \lesssim \|\phi\|_{W^{1,\infty}} \|f\|_{H^{\frac{\beta}{2}}}^2 .    
\end{equation}
In particular, if $u \in C_t H^{\beta}_x(\R\times \R)$, then
\begin{equation}
\sup_{t \in \mathbb{R}} |Z_R(t)| \lesssim_A R .
\end{equation}
\end{lem}
\begin{proof}
    It is enough to show $\|\phi f\|_{H^{\frac{\beta}{2}}}\lesssim \|\phi\|_{W^{1,\infty}} \|f\|_{H^{\frac{\beta}{2}}}$. Note that 
    \begin{equation}
        \|\phi f\|_{L^2}\lesssim \|\phi\|_{L^\infty} \|f\|_{L^2},
    \end{equation}
    and, by the Leibniz rule,
    \begin{align}
        \|\phi f\|_{H^1} \lesssim \|\phi\|_{W^{1,\infty}} \|f\|_{H^1}.
    \end{align}
    Complex interpolation then gives the result, since $0\leq\frac\beta2\leq 1$.
\end{proof}

Upon differentiating \eqref{Z_def} in time, one has
\begin{equation}
    Z_R'(t) = \underbrace{\iint \psi_R(y-x) \rho(t,x)\,\partial_te(t,y)\, dy dx}_{Z_1(t)} + \underbrace{\iint \psi_R(y-x) \partial_t\rho(t,x)\,e(t,y)\, dx dy}_{Z_2(t)} . 
\end{equation}

To proceed, we claim the following technical results, which provide control of $Z_R^\prime(t)$.

\begin{lem}\label{lem_Z1_bound}
   For all $\varepsilon >0$, there exists $C_\varepsilon >0$ such that the following bound on $Z_1$ holds, for all $t\geq 0$:
   \begin{align}
    Z_1(t) \geq{(1-\varepsilon)} \iint \rho[\varphi_R(x-\ell)u(t,x)]
\, dx \int \k[\varphi_{R}(y-\ell)u(t,y)]\,dy \, \frac{d\ell}{R} + O_{A}\left(\varepsilon+\frac{C_{\varepsilon}}{R^{1/4}}\right).
\end{align}
\end{lem}

\begin{lem}\label{lem_Z2_bound}
   For all $\varepsilon >0$, the following bound on $Z_2$ holds, for all $t \geq 0$:
   \begin{equation}
       Z_2(t) = -(1-\varepsilon)\iint e[\varphi_R(y-\ell)u(t,x)]\, dx \int j[\varphi_R(x-\ell)u(t,y)]\, dy \, \frac{d\ell}{R}+O_{A}\left(\varepsilon +\frac{1}{ R^{1/4}}\right).
   \end{equation}
\end{lem}

For the sake of readability, we split the proofs of Lemmas \ref{lem_Z1_bound} and \ref{lem_Z2_bound} between Sections \ref{proof_lem_Z1_bound} and \ref{proof_lem_Z2_bound}, respectively.

\begin{re}
    In \cite{MR3772197} and \cite{MR4009456}, the arguments in the localization use a quadratic difference formula
    \begin{equation}
        \int \phi^2 (D^{\alpha} u)^2 = \int (D^{\alpha}(\phi u))^2 + O(\|\phi'\|_{H^4}^2\|u\|_{H^{\frac{\alpha}{2}}}^2)
    \end{equation}
    for $\alpha \in \{1,2\}$. The proof relies on integrating the squares
    \begin{equation}
        (\phi D^{\alpha} u)^2  = (D^{\alpha}(\phi u)-[D^{\alpha},\phi]u)^2
    \end{equation}
    and treating the cross term by either integrating by parts, in the KdV case, or by the Bajšanski-Coifman smoothing estimate \cite{BC1966} (see also Dawson et al. \cite{DawsonPonce}) for the Calderón commutator in the BO case, thus transferring enough derivatives to the (smooth) weight $\phi$. The same principle, however, does not naturally extend to the case of fractional dispersion. 
    In principle, one could deduce an analogous formula for $\alpha \in (1,2)$, by expanding the singular integral formulation \eqref{eq:gagliardo_singular_integral} to compute $\phi ^2 (D^{\alpha}u)^2 - (D^{\alpha}(\phi u))^2$ in terms of lower-order derivatives in $u$. Nonetheless, we employ a simpler solution, allowing an arbitrarily small loss in the monotonicity formula, which can readily be absorbed by the localization given by compactness.

\end{re}

\subsection{Error Estimates}
Before proving the lemmas, we recall two commutator identities due to Ginibre and Velo from their seminal works \cite{MR1122309, MR1033618}.

\begin{lem}\label{lema_comutador_definitivo}
Let $0<\beta<2$ and $\phi \in C^\infty(\R)$, with $\phi^\prime \in C^\infty_0(\R)$. Then, one has the commutator expansions:
\begin{itemize}
    \item[(a)] If $0<\beta\leq 1$:
\begin{align}
    [D^{2\beta}\partial_x, \phi]f&=\frac{2\beta+1}{2}  D^\beta(\phi^\prime D^\beta_xf)  - \frac{2\beta+1}{2}\mathcal{H} D^\beta(\phi^\prime D^\beta\mathcal{H}f)  + R_\beta[\phi]f;
\intertext{\item[(b)] If $1<\beta<2$:}
    [D^{2\beta}\partial_x, \phi]f&=\frac{2\beta+1}{2}  D^\beta(\phi^\prime D^\beta_xf) - \frac{2\beta+1}{2} D^\beta\mathcal{H}(\phi^\prime D^\beta\mathcal{H}f) - \frac{\beta(\beta-1)(2\beta+1)}{12}D^{\beta-1}(\phi^{\prime \prime \prime}D^{\beta-1}f)\\ 
& \quad  + \frac{\beta(\beta-1)(2\beta+1)}{12}D^{\beta-1}\mathcal{H}(\phi^{\prime \prime \prime}D^{\beta-1} \mathcal{H}f) + R_\beta[\phi]f,
\end{align}
\end{itemize}
where $R_\beta$ satisfies
\begin{equation}
    \|R_\beta[\phi]f\|_{L^2} \lesssim \|\widehat{(D^{\beta +2}\phi)}\|_{L^1} \|f\|_{L^2}.
\end{equation}
\end{lem}

\begin{re}
Observe that, since we can write 
\begin{equation}\label{formula s>1}
    \int \phi\, f\, D^{2\beta} \partial_x f\, dx
   =-\frac{1}{2} \int f\, [D^{2\beta}\partial_x, \phi]\, f\, dx
\end{equation}
we obtain the follwing formula
\begin{equation}\label{identity_s_leq_1}
    \int \phi\, f\, D^{2\beta} \partial_x f\, dx
    = -\frac{2\beta+1}{4} \int \phi' \bigl( |D^\beta f|^2 + |D^\beta \mathcal{H} f|^2 \bigr)\, dx
      + \int f\, R_\beta[\phi]\, f\, dx,
\end{equation}
for all $\beta \in (0,1]$.  Likewise, a similar identity remains valid for $\beta \in [1,2)$:
\begin{align}\label{identity_s_geq_1}
    \int \phi\, f\, D^{2\beta} \partial_x f\, dx
    &= -\frac{2\beta+1}{4} \int \phi' \bigl( |D^\beta f|^2 + |D^\beta \mathcal{H} f|^2 \bigr)\, dx\\
    &\quad+\frac{\beta(\beta-1)(2\beta+1)}{24} \int \phi^{\prime \prime \prime} \bigl( |D^{\beta-1} f|^2 + |D^{\beta-1} \mathcal{H} f|^2 \bigr)\, dx\\
      &\quad+ \int f\, R_\beta[\phi]\, f\, dx.
\end{align}
\end{re}
\begin{re}
    Lemma \ref{lema_comutador_definitivo} can be used as a replacement for \eqref{identidade} if the solution $u$ does not decay fast at infinity. The terms with higher-order derivatives in $\phi$ are due to space truncation, and will be shown to decay as $R$ increases.
\end{re}



We recall the following estimate due to Li, which will be used in the estimates below.

\begin{lem}[cf. {\cite[Corollary 1.4]{Li2019}}]\label{lem:Li-corollary-1-4}
Let $1<p<\infty$ and $s>0$. Assume that $A^s$ is a Fourier multiplier with symbol
$A^s(\xi)$ homogeneous of degree $s$ and smooth on $\mathbb{S}^{d-1}$. Let
$s_1,s_2\geq 0$ be such that $s_1+s_2=s$. Then, for all
$f,g\in \mathcal{S}(\mathbb{R}^d)$, the following estimates hold:

If $1<p_1,p_2<\infty$ and
\[
\frac1p=\frac1{p_1}+\frac1{p_2},
\]
then
\[
\left\|
A^s(fg)
-\sum_{|\alpha|\leq s_1}\frac{\partial^\alpha f}{\alpha!}A^s_\alpha g
-\sum_{|\beta|\leq s_2}\frac{\partial^\beta g}{\beta!}A^s_\beta f
\right\|_{L^p}
\lesssim
\|D^{s_1}f\|_{L^{p_1}}\|D^{s_2}g\|_{L^{p_2}}.
\]

If $p_1=p$ and $p_2=\infty$, then
\[
\left\|
A^s(fg)
-\sum_{|\alpha|<s_1}\frac{\partial^\alpha f}{\alpha!}A^s_\alpha g
-\sum_{|\beta|\leq s_2}\frac{\partial^\beta g}{\beta!}A^s_\beta f
\right\|_{L^p}
\lesssim
\|D^{s_1}f\|_{L^p}\|D^{s_2}g\|_{\mathrm{BMO}}.
\]

If $p_1=\infty$ and $p_2=p$, then
\[
\left\|
A^s(fg)
-\sum_{|\alpha|\leq s_1}\frac{\partial^\alpha f}{\alpha!}A^s_\alpha g
-\sum_{|\beta|<s_2}\frac{\partial^\beta g}{\beta!}A^s_\beta f
\right\|_{L^p}
\lesssim
\|D^{s_1}f\|_{\mathrm{BMO}}\|D^{s_2}g\|_{L^p}.
\]

Here, for each multi-index $\alpha$, the operator $A^s_\alpha$ is defined by
\[
\widehat{A^s_\alpha h}(\xi)
=
i^{-|\alpha|}
\partial_\xi^\alpha A^s(\xi)\widehat h(\xi).
\]
\end{lem}

As a consequence, we obtain the following commutator estimates.

\begin{lem}[cf. {\cite[Theorem 1.2 and Corollary 1.4]{Li2019}}]\label{lem:Li-commutator}
Let $\beta\in(0,1)$ and $1<p<\infty$. Then
\[
\|[D^\beta,\phi]f\|_{L^p}
+
\|[\mathcal{H}D^\beta,\phi]f\|_{L^p}
\lesssim
\|D^\beta \phi\|_{L^\infty}
\|f\|_{L^p},
\]
and
\[
\|[D^\beta,\phi]f\|_{L^p}
+
\|[\mathcal{H}D^\beta,\phi]f\|_{L^p}
\lesssim
\|D^\beta f\|_{L^p}
\|\phi\|_{L^\infty}.
\]
\end{lem}


The following lemma establishes a functional identity involving $\psi_R$ and $\varphi_R$ (see \eqref{eq:def_psi_R}).
\begin{lem}\label{lem:claim_tilde_Phi} Let $u \in \mathcal{S}(\R)$ and $\beta \in (1,2]$. Then, we have the identity
\begin{equation}\label{claim_tilde_Phi}
    \int_{\R} \psi_R(y-x)\, u(x)\, D_x^{\beta}\partial_x u(x)\, dx
    =
    -\iint_{\R\times\R}
        \varphi_R^{\,2}(y-\ell)\,
        \widetilde{\Psi}_R(x-\ell)\,
        u(x)\, D_x^{\beta}\partial_x u(x)\,
        dx\,\frac{d\ell}{R},
\end{equation}
where
\begin{equation}\label{def_tildePsi}
    \widetilde{\Psi}_R(x)
    := \int_{-\infty}^{x} \varphi_R^{\,2}(\ell)\, d\ell,
\end{equation}
is a primitive for $\varphi_R^{\,2}$.
\end{lem}

\begin{proof} If $\Psi_R$ and $U$ are the $s$-order fractional extensions of $\psi_R$ and $u$, respectively, by the Stokes  theorem, we have 
\begin{align}
    \int_{\R} \psi_R(y-x) u(x) D^\beta_x \partial_x u(x) \,dx &= - c\iint_{\R \times \R_+} \nabla_{x,z}(\Psi_R(y-x,z)U(x,z) )\cdot \nabla_{x,z}\partial_x U(x,z) z^{1-\beta}\, dx\,dz\\
    &\hspace{-40mm}= - c\iint_{\R \times \R_+} \nabla_{x,z}[\Psi_R(y-x,z)] \cdot \nabla_{x,z}\partial_x U(x,z) U(x,z)z^{1-\beta}\, dx\,dz\\
    & - c\iint_{\R \times \R_+} \Psi_R(y-x,z)\nabla_{x,z}[U(x,z) ]\cdot \nabla_{x,z}\partial_x U(x,z) z^{1-\beta}\, dx\
    dz\\
    &\hspace{-40mm}= - c\iint_{\R \times \R_+} \nabla_{x,z}[\Psi_R(y-x,z)] \cdot \nabla_{x,z}\partial_x U(x,z)\, U(x,z)\, z^{1-\beta}\, dx\,dz\\
    & + \frac{c}{2}\iint_{\R \times \R_+} \partial_x\Psi_R(y-x,z) |\nabla_{x,z} U(x,z)|^2 \,z^{1-\beta}\, dx\,dz\\
    &\hspace{-40mm}=c\iint_{\R\times \R_+}\left(P_z*\psi_R'(y-x),\tfrac{x}{z}P_z*\psi_R'(y-x)\right) \cdot \nabla_{x,z}\partial_x U(x,z)\, U(x,z)\, z^{1-\beta}\,dxdz\\
    &+ \frac{c}{2}\iint_{\R \times \R_+} P_z*\psi_R'(y-x) |\nabla_{x,z} U(x,z)|^2 \,z^{1-\beta}\, dx\,dz\\
    &\hspace{-40mm}=c\int \varphi^2_R(y-\ell)\iint_{\R\times \R_+}\left(P_z*\varphi^2_R(x-\ell),\tfrac{x}{z}P_z*\varphi^2_R(x-\ell)\right) \cdot \nabla_{x,z}\partial_x U(x,z)\, U(x,z)\, z^{1-\beta}\,dxdz\, \frac{d\ell}{R}\\
    &+ \frac{c}{2}\int \varphi^2_R(y-\ell)\iint_{\R \times \R_+} P_z*\varphi_R^2(x-\ell) |\nabla_{x,z} U(x,z)|^2 \,z^{1-\beta}\, dx\,dz\\
    &\hspace{-40mm}=-\iint \varphi_R^2(y-\ell) \tilde{\Psi}_R(x-\ell) u(x) D^\beta_x \partial_x u(x) \, dx \, \frac{d\ell}{R}.
\end{align}
\end{proof}
\subsubsection{Proof of Lemma \ref{lem_Z1_bound}}
\label{proof_lem_Z1_bound}
\begin{proof}
The proof of this lemma will be organized into a sequence of steps:

 \underline{Step 1}: \emph{Obtaining estimates that only involve $\psi'_R$:}
    
    Following the discussion in Section 4, we rely on integration by parts and commutator estimates to obtain a truncated version of the monotonicity formula. We expect at least one derivative of $\psi_R$ in all terms, due to the skew-symmetry of the operator $D^\alpha \partial_x$.  For the sake of readability, we often omit the $t$ variable throughout the proof.
    
    By direct computation:
\begin{align}
    Z_1 = &\label{7.9}-\frac{1}{2}\iint \psi_R(y-x) \rho(x)\,
 D^{\alpha}_y\partial_y u(y) \, D^\alpha_y u(y)\, dy\, dx\\
 &-\frac{1}{2}\iint \psi_R(y-x) \rho(x)\,
u(y)\, D^{2\alpha}_y\partial_yu(y)\, dy dx\\ 
&-\frac{1}{2}
    \iint \psi_R(y-x) \rho(x)\,\partial_y(u^{k+1})(y)\, D^{\alpha}_y u(y)\, dy dx \\
&-\frac{1}{2}
    \iint \psi_R(y-x) \rho(x)\,u(y)\, D^{\alpha}_y\partial_y (u^{k+1})(y)\, dy dx \\
&-\iint \psi_R(y-x) \rho(x)\,u^{k+1}(y)\,D^{\alpha}_y \partial_y u (y)\, dy dx \\ 
&\label{7.13}-\iint \psi_R(y-x) \rho(x)\,u^{k+1}(y)\partial_y(u^{k+1})(y)\, dy dx.
\end{align}

Integrating \eqref{7.9} and \eqref{7.13} by parts, we obtain
\begin{align}
\label{A}Z_1
&=\frac{1}{4}\iint \psi_R'(y-x) \rho(x)\,
 (D^{\alpha}_yu)^2(y) \,dy\, dx\\
& \label{B}
 \quad-\frac{1}{2}\iint \psi_R(y-x) \rho(x)\,
u(y)\, D^{2\alpha}_y\partial_yu(y)\, dy \,dx\\ 
&\quad-\frac{1}{2}
    \iint \psi_R(y-x) \rho(x)\,\partial_y(u^{k+1})(y)\, D^{\alpha}_y u(y)\, dy \,dx\label{C}\\
&\quad-\frac{1}{2}
    \iint \psi_R(y-x) \rho(x)\,u(y)\, D^{\alpha}_y\partial_y (u^{k+1})(y)\, dy\, dx \label{D}\\
&\quad-\iint \psi_R(y-x) \rho(x)\,u^{k+1}(y)\,D^{\alpha}_y \partial_y u(y) \, dy\, dx \label{E} \\ 
&\quad +\frac{1}{2}\iint \psi_R'(y-x) \rho(x)\,u^{2k+2}(y)\, dy\, dx.
\end{align}

Estimating term \eqref{B} with Lemma \ref{lema_comutador_definitivo}, we have:
\begin{align}
\eqref{B}
&= \frac{2\alpha+1}{8} \iint \psi^\prime_R(y-x) \rho(x)   \left[(D^\alpha_yu)^2(y)+(\mathcal H_yD^\alpha_y u)^2(y)\right](y)\, dy\,dx \\
&\quad-\frac{\alpha(\alpha-1)(2\alpha+1)}{48} \iint \psi^{\prime \prime \prime}_R(y-x)\rho(x)  [(D^{\alpha-1}_y u)^2(y)+(\mathcal H_y D^{\alpha-1}_yu)^2(y) ]\,dy\,dx \\
& \quad - \frac{1}{4} \iint \rho(x)  u(y) R_{\alpha}[\psi_R(\cdot-x)] u (y)\, dy\,dx.
\end{align}







We now work on \eqref{C}, \eqref{D} and \eqref{E} 
Note that, via integration by parts,
\begin{equation}
    \eqref{C} = \frac{1}{2}
    \iint \psi_R'(y-x) \rho(x)\,u^{k+1}(y)\, D^{\alpha}_y u(y)\, dy\, dx -\frac{1}{2}\eqref{E}, 
\end{equation}
and by Lemma \ref{lema_comutador_definitivo},
\begin{align}
    \eqref{D} &=  \frac{1}{2}
    \iint  \rho(x)\, D^{\alpha}_y\partial_y (\psi_R(\cdot-x) u)(y)\,u^{k+1}(y)\, dy\, dx\\
    &= \frac{\alpha+1}{4}\iint \psi'_R(y-x)\rho(x)D^{\frac{\alpha}{2}}_yu(y)D^{\frac{\alpha}{2}}_y(u^{k+1})(y)\,dy\,dx\\
    &\quad -\frac{\alpha+1}{4} \iint\psi'_R(y-x) \rho(x)  D_y^{\frac{\alpha}{2}} \mathcal{H}_yu(y) D_y^{\frac{\alpha}{2}} \mathcal{H}_y(u^{k+1})(y)\, dy\, dx \\
    &\quad  +\frac{1}{2}\iint \rho(x) R_{\frac{\alpha}{2}}[\psi_R(\cdot-x)]u(y) u^{k+1}(y)\,dy\,dx - 
    \frac{1}{2}\eqref{E}.
\end{align}
Therefore, we conclude that
\begin{align}
    Z_1 &=\label{Z1_A1}\tag{$Z_{1,1}$}   \frac{2\alpha+3}{8} \iint \psi^\prime_R(y-x)\rho(x)   (D^\alpha_yu)^2(y)\, dy\,dx,  \\
    \label{Z1_A2} \tag{$Z_{1,2}$} &  \quad+\frac{2\alpha+1}{8} \iint \psi^\prime_R(y-x)\rho(x)   (\mathcal H D^\alpha_yu)^2(y)\, dy\,dx,\\
   \label{Z1_B1}\tag{$Z_{1,3}$}&\quad+\frac{1}{2}
    \iint \psi_R'(y-x) \rho(x)\,u^{k+1}(y)\, D^{\alpha}_y u(y)\, dy\, dx\\
    \label{Z1_C1}\tag{$Z_{1,4}$}   &\quad+\frac{\alpha+1}{4}\iint \psi'_R(y-x)\rho(x)D^{\frac{\alpha}{2}}_yu(y)D^{\frac{\alpha}{2}}_y(u^{k+1})(y)\,dy\,dx\\
    \label{Z1_D1}\tag{$Z_{1,5}$}&\quad- \frac{\alpha+1}{4} \iint  \psi^\prime_R(y-x) \rho(x) D_y^{\frac{\alpha}{2}} \mathcal{H}_yu(y) D_y^{\frac{\alpha}{2}} \mathcal{H}_y(u^{k+1})(y)\, dy\, dx \\
    \label{Z1_E1}\tag{$Z_{1,6}$}&\quad+\frac{1}{2}\iint \psi_R'(y-x) \rho(x)\,u^{2k+2}(y)\, dy\, dx\\
    &\quad+Z_{1,e},
    \end{align}
where we claim that 
    \begin{align}
Z_{1,e}&:= -\frac{\alpha(\alpha-1)(2\alpha+1)}{48} \iint \psi^{\prime \prime \prime}_R(y-x) \rho(x)   (D^{\alpha-1}_y u )^2(y) \,dy\,dx \\
&\quad-\frac{\alpha(\alpha-1)(2\alpha+1)}{48} \iint  \psi^{\prime \prime \prime}_R(y-x)\rho(x)(D^{\alpha-1}_y\mathcal{H}u)^2(y) \,dy\,dx\\
&\quad  +\frac{1}{2}\iint \rho(x)u^{k+1}(y) R_{\frac{\alpha}{2}}[\psi_R(\cdot-x)]u(y) \,dy\,dx\\& \quad - \frac{1}{4} \iint \rho(x)  u(y) R_{\alpha}[\psi_R(\cdot-x)] u (y)\, dy\,dx,
\end{align}
is an error term. Indeed, since 
\begin{equation}
 \|\psi'''_R\|_{L^\infty} + \sup_x\left\|R_\alpha[\psi_R(\cdot-x)]\right\|_{L^2 \to L^2}+\sup_x\left\|R_{\frac{\alpha}{2}}[\psi_R(\cdot-x)]\right\|_{L^2 \to L^2}\lesssim \frac{1}{R^{1/4}}   
\end{equation}
and
\begin{equation}
    \|\rho(x)\|_{L^\infty_tL^1_x} + \|D^{\alpha-1}_yu(y)\|_{L^{\infty}_tL^2_y}+ \|u(y)\|_{L^{\infty}_{t,y}}\lesssim_A 1,
\end{equation}
we have
\begin{equation}
    Z_{1,e}=O_{A}\left( \frac{1}{R^{1/4}}\right).
\end{equation}
\underline{Step 2}: \emph{Commuting localization and fractional derivatives:}

Following \cite{MR4009456}, we write $\varphi_{R,\ell} := \varphi_R(\cdot - \ell)$ and abbreviate
\begin{equation}
    \iiiint_{x,y,\ell,z} := \iiiint_{\mathbb{R}^3 \times \mathbb{R}_+} dx\,dy\,\frac{d\ell}{R}\,dz,
\end{equation}
with analogous conventions for triple, double and single integrals.
 
Here, we aim to show that
\begin{align}
    Z_1(t) \geq (1-\varepsilon)
    \iiint_{x,y,\ell}  \rho[\varphi_{R,\ell}u](x)
  \k[\varphi_{R,\ell}u](y) + O_{A}\left(\varepsilon+\frac{C_\varepsilon}{R^{1/4}}\right).
\end{align}
That is, the goal is to switch from $\varphi_{R,\ell}^2\, \rho[u]$ and $\varphi_{R,\ell}^2\, \k[u]$ to $\rho[\varphi_{R,\ell}\, u]$ and $\k[\varphi_{R,\ell}\, u]$, respectively. 

We start by recalling \eqref{psi_prime} and writing the integral in \eqref{Z1_A1} as
\begin{align}
    \iint_{x,y} \psi_R'(y-x) \rho(x)\,
 (D^{\alpha}_yu)^2(y)  = \iiint_{x,y,\ell} (\varphi_{R,\ell} u)^2(x) (\varphi_{R,\ell}
 D^{\alpha}_yu)^2(y)
\end{align}


Now, by Lemma \ref{lem:Li-commutator}, together with Young's inequality and the immediate bound
\begin{equation}\label{commutator_argenis}
    \sup_{\ell \in \mathbb R}\|[D^{\alpha};\varphi_{R,\ell}]u\|_{L^2} \lesssim_A \frac{1}{R^{1/4}},
\end{equation}  
we have
\begin{align}
 \int_y (\varphi_{R,\ell}(y)
 D^{\alpha}_yu)^2(y) &=
 \int_y (D^{\alpha}_y(\varphi_{R,\ell}
 u))^2(y)
 -2 \int_y D^{\alpha}_y \left(\varphi_{R,\ell}
 u \right)[D^\alpha_y; \varphi_{R,\ell}]u(y) + \int_y \left([D^\alpha_y;\varphi_{R,\ell}]u\right)^2(y) \\
 &\geq (1-\varepsilon)\int_y (D^{\alpha}_y(\varphi_{R,\ell}
 u))^2(y)- C_{\varepsilon}  \int_y \left([D^\alpha_y; \varphi_{R,\ell}]u\right)^2(y) \\
 &= (1-\varepsilon)\int_y(D^{\alpha}_y(\varphi_{R,\ell}
 u))^2(y) + O_{A}\left(\frac{C_{\varepsilon}}{R^{1/4}}\right).
\end{align}

We then estimate \eqref{Z1_A1} as
\begin{align}
    \eqref{Z1_A1} 
 &\geq  \frac{(1-\varepsilon)(2\alpha+3)}{8}\iiint_{x,y,\ell}  (\varphi_{R,\ell} u)^2(x)\,(D^{\alpha}_y(\varphi_{R,\ell}
 u))^2(y))+  O_{A}\left(\frac{C_\varepsilon}{R^{1/4}}\int u^2(x)\left(\int \varphi^2_R(x-\ell)\,\frac{d\ell}{R}\right)\,dx \right)\\
 &\label{final_Z1_A1}=  \frac{(1-\varepsilon)(2\alpha+3)}{8}\iiint_{x,y,\ell}  (\varphi_{R,\ell} u)^2(x)\,(D^{\alpha}_y(\varphi_{R,\ell}
 u))^2(y)+ O_{A}\left(\frac{C_\varepsilon}{R^{1/4}}\right).
\end{align}

Similarly, again due to Lemma \ref{lem:Li-commutator} 
\begin{align}\label{commutator_hilbert_smooth}
        \sup_{\ell \in \mathbb R}\|[\mathcal{H}D^{\alpha};\varphi_{R,\ell}]u\|_{L^2} \lesssim_A \frac{1}{R^{1/4}},
\end{align}
we have, for \eqref{Z1_A2}
\begin{equation}\label{final_Z1_A2}
    \eqref{Z1_A2} \geq  \frac{(1-\varepsilon)(2\alpha+1)}{8}\iiint_{x,y,\ell}  (\varphi_{R,\ell} u(x))^2\,(D^{\alpha}_y(\varphi_{R,\ell}
 u(y)))^2+ O_{A}\left(\frac{C_\varepsilon}{R^{1/4}}\right)
\end{equation}

We now estimate the integral in \eqref{Z1_B1} by using the \textit{almost} non-negativity estimate \eqref{almost_non_negativity}:
\begin{align}\label{Z1_C1_interm1}
    \iint_{x,y} \psi_R'(y-x) \rho(x)\,u^{k+1}(y)\, D^{\alpha}_y u(y) &=\iiint_{x,y,\ell}   \varphi^2_R(x-\ell) \rho(x)\,\varphi^2_R(y-\ell)u^{k+1}(y)\, D^{\alpha}_y u(y)\\
    & \geq  \int_{\ell} \left(\int_x \varphi^2_{R,\ell}(x)u^2(x)\right) \, \left(\int_y \varphi^{k+1}_{R,\ell}(y)u^{k+1}(y)\, D^{\alpha}_y u(y)\,\right)  \\&\qquad + O_{A}\left(\|D^{\alpha}(\varphi^2_R-\varphi^{k+1}_R)\|_{L^\infty}\right)\\
    &=   \int_{\ell} \left(\int_x (\varphi_{R,\ell}u)^2(x) \right) \, \left(\int_y(\varphi_{R,\ell}u)^{k+1}(y)\, D^{\alpha}_y (\varphi_{R,\ell}u)(y)\right)\\
    &\quad -\int_x u^2(x)\left(\int_{\ell} \varphi_{R}^2(x-\ell)  \, \left(\int_y(\varphi_{R,\ell}u)^{k+1}(y)\, [D^{\alpha}_y;\varphi_{R,\ell}]u(y)\right) \right) \\&\quad+ O_{A}\left(\frac{1}{R^{1/4}}\right)\\
    &\label{C+D+E} =  \iiint_{x,y,\ell} (\varphi_{R,\ell}u)^2(x)(\varphi_{R,\ell}u)^{k+1}(y)\, D^{\alpha}_y (\varphi_{R,\ell}u)(y)\\
   &\qquad + O_{A}\left(\frac{1}{R^{1/4}}\right),
\end{align}
since, for all $s\in \R$,
\begin{align}
    \left|\int_y(\varphi_{R,\ell}u)^{k+1}(y)\, [D^{\alpha}_y;\varphi_{R,\ell}]u(y)\right| \leq \|u^{k+1}\|_{L^2}\|[D^{\alpha};\varphi_{R,\ell}]u\|_{L^2}\lesssim_A  \frac{1}{R^{1/4}}.
\end{align}
We then conclude, for \eqref{Z1_B1},
\begin{equation}\label{final_Z1_B1}
        \eqref{Z1_B1} \geq \frac{1}{2} \eqref{C+D+E}  + O_{A}\left(\frac{1}{R^{1/4}}\right).
\end{equation}
To estimate \eqref{Z1_C1}, we write
\begin{align}
    \iint_{x,y}  \psi^\prime_R(y-x)\rho(x) D_y^{\frac{\alpha}{2}} u(y) D_y^{\frac{\alpha}{2}} (u^{k+1})(y)=\iint_{x,y,s}  (\varphi_{R,\ell}u)^2(x)\varphi_{R,\ell}^2(y) D_y^{\frac{\alpha}{2}} u(y) D_y^{\frac{\alpha}{2}} (u^{k+1})(y).
\end{align}
Now, since
\begin{align}
    \int_y\varphi_{R,\ell}^2(y) D_y^{\frac{\alpha}{2}} u(y) D_y^{\frac{\alpha}{2}} (u^{k+1})(y) &= \int_y D_y^{\frac{\alpha}{2}}(\varphi_{R,\ell}^2 D_y^{\frac{\alpha}{2}} u)(y)  \,u^{k+1}(y) \\
    &=\int_y \varphi_{R,\ell}^2(y) D_y^{\alpha} u(y)\,  u^{k+1}(y) + \int_y  u^{k+1}(y) \,[D^{\frac{\alpha}{2}};\varphi_{R,\ell}] D_y^{\frac{\alpha}{2}}u(y),
\end{align}
we have
\begin{equation}\label{final_Z1_C1}
    \eqref{Z1_C1}\geq \frac{\alpha+1}{4}\eqref{C+D+E}+O_{A}\left(\frac{1}{R^{1/4}}\right).
\end{equation}
Analogously,
 \begin{equation}\label{final_Z1_D1}
     \eqref{Z1_D1}\geq \frac{\alpha+1}{4}\eqref{C+D+E} + O_{A}\left(\frac{1}{R^{1/4}}\right).
 \end{equation}
Finally, since
 \begin{align}
     \iint_{x,y} \psi_R'(y-x) \rho(x)\,u^{2k+2}(y) = \iiint_{x,y,\ell} \varphi_{R,\ell}^2(x) u^2(x) \varphi_{R,\ell}^2 (y)u^{2k+2}(y) \geq \iiint_{x,y,\ell} (\varphi_{R,\ell} u)^2(x) (\varphi_{R,\ell} u)^{2k+2}(y),
 \end{align}
one can bound \eqref{Z1_E1} as
\begin{equation}\label{final_Z1_E1}
    \eqref{Z1_E1} \geq \frac{1}{2}\iiint_{x,y,\ell} (\varphi_{R,\ell} u)^2(x) (\varphi_{R,\ell} u)^{2k+2}(y).
\end{equation}
    Combining the estimates \eqref{final_Z1_A1}, \eqref{final_Z1_A2}, \eqref{final_Z1_B1}, \eqref{final_Z1_C1}, \eqref{final_Z1_D1}, \eqref{final_Z1_E1}, and the immediate bound
    \begin{equation}
        |\eqref{C+D+E}| + |\eqref{final_Z1_E1}| \lesssim_A 1
    \end{equation}
    we obtain
\begin{align}
    Z_1 \geq (1-\varepsilon)\iiint_{x,y,\ell} \rho[\varphi_{R,\ell}u](x)
  \k[\varphi_{R,\ell}u](y) + O_{A}\left(\varepsilon+\frac{C_\varepsilon}{R^{1/4}}\right).
\end{align}
\end{proof}
\subsubsection{Proof of Lemma \ref{lem_Z2_bound}}
\label{proof_lem_Z2_bound}

As in Lemma \ref{lem_Z1_bound}, we divide the proof into two steps. In contrast with $\rho = u^2$, here the integral of $|e|$ might not be bounded, since $e$ is not necessarily pointwise nonnegative, thus requiring us to invoke Lemma \ref{localization_lemma}. Moreover, passing from $\varphi_{R,\ell}^2e[u]$ to $e[\varphi_{R,\ell}u]$ involves employing more commutator estimates.

\underline{Step 1}:
Observe that, by the formula \eqref{claim_tilde_Phi} and the Lemma \ref{lema_comutador_definitivo},
\begin{align}
Z_2(t) 
&= -2 \iint_{x,y} \psi_R(y-x)\, e(y)   \, u(x)D^\alpha_x \partial_x u(x)- \frac{2(k+1)}{k+2} \iint_{x,y} \psi_R(y-x) \,e(y)   \, \partial_x(u^{k+2})(x)  \\
&= 2  \iiint_{x,y,\ell} \varphi_{R,\ell}^2(y)\, e(y)\tilde{\Psi}_R(x-\ell) u(x) D^\alpha_x \partial_x u(x) -\frac{2(k+1)}{k+2} \iint_{x,y}\psi_R^\prime(y-x) e(y)   \, u^{k+2}(x) \\
&\label{Z_21}\tag{$Z_{2,1}$}=-\frac{\alpha+1}{4} \iiint_{x,y,\ell}  \varphi_{R,\ell}^2(y)u(y)D^{\alpha}_y u(y)   \, 
(\varphi_{R,\ell}D^{\frac{\alpha}{2}}_x u)^2 (x)  \\
& \label{Z_22}\tag{$Z_{2,2}$}\quad + \frac{\alpha+1}{4} \iiint_{x,y,\ell} \varphi_{R,\ell}^2(y)u(y)D^{\alpha}_y u(y)   \, 
(\varphi_{R,\ell}\mathcal H D^{\frac{\alpha}{2}}_x u)^2 (x)  \\
&\label{Z_23}\tag{$Z_{2,3}$}\quad-\frac{\alpha+1}{2(k+2)} \iiint_{x,y,\ell}  \varphi_{R,\ell}^2(y)u^{k+2}(y)   \, 
(\varphi_{R,\ell}D^{\frac{\alpha}{2}}_x u)^2 (x)   \\
&\label{Z_24}\tag{$Z_{2,4}$}\quad+\frac{\alpha+1}{2(k+2)} \iiint_{x,y,\ell}  \varphi_{R,\ell}^2(y) u^{k+2}(y)  \, 
(\varphi_{R,\ell}\mathcal H D^{\frac{\alpha}{2}}_x u)^2 (x)  \\
&\label{Z_25}\tag{$Z_{2,5}$}\quad -\frac{k+1}{k+2} \iiint_{x,y,\ell}  \varphi_{R,\ell}^2 u(y)D^{\alpha}_y u(y)\,\varphi_{R,\ell}^2(x) u^{k+2}(x)  \\
&\label{Z_26}\tag{$Z_{2,6}$}\quad -\frac{2(k+1)}{(k+2)^2} \iiint_{x,y,\ell} \varphi_{R,\ell}^2(y) u^{k+2}(y)\, \varphi_{R,\ell}^2(x) u^{k+2}(x) \\
&\label{Z_2e}\tag{$Z_{2,e}$}\quad -   \iiint_{x,y,\ell}  \varphi_{R,\ell}^2(y) e(y) u(x) R_{\frac{\alpha}{2}}(\tilde \Psi_R(\cdot-\ell)) u(x).
\end{align}
We claim that \eqref{Z_2e} is an error term. Indeed, we have that
\begin{equation}
   \left\|\int_{\ell} \varphi_{R,\ell}^2(y)\int_x   \varphi_{R,\ell}^2(x) u(x) R_{\frac{\alpha}{2}}[\tilde \Psi_R(\cdot-\ell)] u(x) \right\|_{W^{1,\infty}_y} \lesssim \frac{\|u_0\|_{L^2}}{R^{1/4}},
\end{equation}
which implies, by the localization Lemma \ref{localization_lemma}
\begin{equation}\label{Z_2e_final}
    |\eqref{Z_2e}| \lesssim_A \frac{1}{R^{1/4}}.
\end{equation}

\underline{Step 2}: We begin by estimating the term \eqref{Z_21}. Note that
\begin{align}
\eqref{Z_21} &=  -\frac{\alpha+1}{4}\iiint_{x,y,\ell} \varphi_{R,\ell}^2(y) u(y) D^\alpha_y u(y) \, 
(\varphi_{R,\ell}D^{\frac{\alpha}{2}}_x u)^2 (x)  \\
&\label{Z_211}\tag{$Z_{2,1,1}$}= -\frac{\alpha+1}{4}\iiint_{x,y,\ell}  (\varphi_{R,\ell} u)(y) \,D^\alpha_y (\varphi_{R,\ell}u)(y)\, 
(D^{\frac{\alpha}{2}}_x (\varphi_{R,\ell}u))^2 (x) \\
&\label{Z_212}\tag{$Z_{2,1,2}$}\qquad+  \frac{\alpha+1}{2}  \iiint_{x,y,\ell}  (D^{\frac{\alpha}{2}}_y (\varphi_{R,\ell}u))^2(y) [D^{\frac{\alpha}{2}}_x; \varphi_{R,\ell}^2]u(x)   \, 
D^{\frac{\alpha}{2}}_x (\varphi_{R,\ell}u) (x) \\
&\label{Z_213}\tag{$Z_{2,1,3}$} \qquad -\frac{\alpha+1}{4}\iiint_{x,y,\ell}  (D^{\frac{\alpha}{2}}_y (\varphi_{R,\ell}u))^2(y) ([D^{\frac{\alpha}{2}}_x; \varphi_{R,\ell}^2]u)^2(x)   \, 
 \\
  &\label{Z_214}\tag{$Z_{2,1,4}$}\qquad +\frac{\alpha+1}{4}\iiint_{x,y,\ell}  \varphi_{R,\ell} u(y)[D^\alpha_y;\varphi_{R,\ell}] u(y)   \, 
(\varphi_{R,\ell}D^{\frac{\alpha}{2}}_x u)^2 (x).
\end{align}
where \eqref{Z_211} will be treated as a main term. To estimate \eqref{Z_212} and \eqref{Z_213}, we recall
\begin{equation}\label{bound_for_A3}
    \sup_{s\in \mathbb R} \|[D^{\frac{\alpha}{2}}; \varphi_{R,\ell}^2]u\|_{L^2} \lesssim _E\frac{1}{R^{1/4}}
\end{equation}
Now, we invoke the fractional extension of order $\frac{\alpha}{2}$ to write
\begin{align}
\iint_{y,s} D^{\frac{\alpha}{2}}_y (\varphi_R(y-\ell)u(y)))^2 &= \iiint_{y,z,\ell}  |\nabla_{y,z} (\Phi_R(y-\ell,z)U(y,z))|^2 \,z^{1-\frac{\alpha}{2}}\\
   &\label{0.101}=\iint_{y,z} |\nabla_{y,z} U(y,z)|^2\,z^{1-\frac{\alpha}{2}}\int_{\ell} \Phi_R^2(y-\ell,z) \, 
  \\ 
  &\label{0.103}\,+\iiint_{y,z,\ell}  U^2(y,z)|\nabla_{y,z} \Phi_R(y-\ell,z)|^2 \,z^{1-\frac{\alpha}{2}}\, 
\\
\label{0.102}&\,+2\iiint_{y,z,\ell} \hspace{-1mm} \left(\Phi_R(y-\ell,z) \nabla_{y,z} U(y,z)\right) \cdot \left(U(y,z)\nabla_{y,z}\Phi_R(y-\ell,z)\right)  \,z^{1-\frac{\alpha}{2}}.
\end{align}
Note that
\begin{align}
    \int_{\ell} \Phi_R^2(y-\ell,z) \lesssim \|\varphi_R\|_{L^\infty}\int_\ell\Phi_R(y-\ell,z)  
=\int_w\left( P_z(w)\int_\ell \varphi_R(y-\ell-w)  \right)
 \lesssim  \int_w P_z(w)\lesssim1.
\end{align}
Thus, \eqref{0.101} is estimated by
\begin{align}
    |\eqref{0.101}|\lesssim \iint_{y,z} |\nabla_{y,z} U(y,z)|^2\,z^{1-\frac{\alpha}{2}}= \|D^{\frac{\alpha}{2}}u\|_{L^2}^2.
\end{align}
We now turn to \eqref{0.103}. Observe, by the change of variables $s = y-\ell$:
\begin{align}
    |\eqref{0.103}| &\lesssim \int_y \left(\mathcal M (u)(y)\right)^2\iint_{z,\ell} |\nabla_{y,z} \Phi_R(y-\ell,z)|^2 \,z^{1-\frac{\alpha}{2}}\\
   &=\int_y \left(\mathcal M (u)(y)\right)^2\iint |\nabla_{s,z} \Phi_R(s,z)|^2 \,z^{1-\frac{\alpha}{2}}\,\frac{ds}{R}\,dz\\
   &=\|D^{\frac{\alpha}{2}}\varphi_R\|^2_{L^2}\int_y \left(\mathcal M (u)(y)\right)^2\,\\
   &\lesssim \frac{\|u_0\|_{L^2}^2}{R^{\alpha-1}}.
\end{align}
Finally, \eqref{0.102} is estimated, via Young's inequality, by $|\eqref{0.101}| + |\eqref{0.103}|$. Note that such terms could not be simply estimated by the fractional Leibniz rule since the necessary localization for the integration in $\ell$ would be lost. 

Therefore, by Cauchy-Schwarz in the $x$-variable, the fractional Leibniz rule and \eqref{bound_for_A3}, we obtain $|\eqref{Z_212}| \lesssim \frac{1}{R^{1/4}}$. In a similar fashion, we have $|\eqref{Z_213}| \lesssim\frac{1}{R^{1/4}}$. Finally, $\eqref{Z_214}$ is estimated by \eqref{commutator_argenis} as 
\begin{align}
    |\eqref{Z_214}| \lesssim  \frac{1}{R^{\alpha}}\int_x \left[(D^{\frac{\alpha}{2}}_x u)^2 (x) \int_{\ell}  \, 
\varphi_{R,\ell}^2  \right] \lesssim_A \frac{1}{R^{1/4}}.
\end{align}
We then conclude
\begin{equation}\label{Z_21_final}
    \eqref{Z_21} = \eqref{Z_211} + O_{A}\left(\frac{1}{R^{1/4}}\right).
\end{equation}
Analogously, by \eqref{commutator_hilbert_smooth}, we have, for $Z_{2,2}$,
\begin{equation}\label{Z_22_final}
    \eqref{Z_22} =  \eqref{Z_211} + O_{A}\left(\frac{1}{R^{1/4}}\right).
\end{equation}
To deal with $Z_{2,3}$, we write
\begin{align}
    \eqref{Z_23}&= -\iiint_{x,y,\ell}  \varphi_{R,\ell}^2(y) u^{k+2}(y) \,(D^{\frac{\alpha}{2}}_x (\varphi_{R,\ell}u) (x)-[D^{\frac{\alpha}{2}};\varphi_{R,\ell}]u(x))^2  \\
    &=-\iiint_{x,y,\ell}  \varphi_{R,\ell}^2(y)\, u^{k+2}(y) \,(D^{\frac{\alpha}{2}}_x (\varphi_{R,\ell}u))^2 (x)\\
    &\quad+ 2 \iiint_{x,y,\ell}  \varphi_{R,\ell}^2(y) u^{k+2}(y) \,D^{\frac{\alpha}{2}}_x (\varphi_{R,\ell}u) (x)[D^{\frac{\alpha}{2}};\varphi_{R,\ell}]u(x) \\
    &\quad - \iiint_{x,y,\ell}  \varphi_{R,\ell}^2(y)\, u^{k+2}(y) ([D^{\frac{\alpha}{2}};\varphi_{R,\ell}]u)^2(x) \\
    &\label{Z_231}\tag{$Z_{2,3,1}$}=-  \iiint_{x,y,\ell}  (\varphi_{R,\ell}u)^{k+2}(y) \,(D^{\frac{\alpha}{2}}_x (\varphi_{R,\ell}u))^2 (x) \\
    &\label{Z_232}\tag{$Z_{2,3,2}$}\quad +2 \iiint_{x,y,\ell}  \varphi_{R,\ell}^2(y) u^{k+2}(y) \,D^{\frac{\alpha}{2}}_x (\varphi_{R,\ell}u) (x)[D^{\frac{\alpha}{2}};\varphi_{R,\ell}]u(x)  \\
    &\label{Z_233}\tag{$Z_{2,3,3}$}\quad - \iiint_{x,y,\ell}  \varphi_{R,\ell}^2(y)\, u^{k+2}(y) ([D^{\frac{\alpha}{2}};\varphi_{R,\ell}]u)^2(x)  \\
    &\label{Z_234}\tag{$Z_{2,3,4}$}\quad -\iiint_{x,y,\ell}  (\varphi_{R,\ell}^2-\varphi_{R,\ell}^{k+2})(y) u^{k+2}(y) \,(D^{\frac{\alpha}{2}}_x (\varphi_{R,\ell}u))^2 (x) ,
\end{align}
with \eqref{Z_231} being the main term. We then estimate the error terms which contain commutators by
\begin{equation}
    |\eqref{Z_232}| + |\eqref{Z_233}| \lesssim_A \frac{1}{R^{1/4}}.
\end{equation}
To estimate $\eqref{Z_234}$, we rely on the support properties of $\varphi_R$ to write
\begin{equation}\label{support_phi_bound}
    \int_{\ell} [\varphi^2_R(\ell)-\varphi^{k+2}_R(\ell)] = \int_{R}^{R + R^{3/4}} [\varphi^2_R(\ell)-\varphi^{k+2}_R(\ell)] \, \frac{d\ell}{R} \lesssim \frac{1}{R^{1/4}}.
\end{equation}
Therefore,
\begin{align}
    |\eqref{Z_234}| &\leq \int_y u^{k+2}(y)\int_{\ell}  (\varphi_{R,\ell}^2-\varphi_{R,\ell}^{k+2})  \,\|D^{\frac{\alpha}{2}} (\varphi_{R,\ell}u)\|_{L^2}^2 \lesssim \|u\|_{L^\infty_tH^{\frac{\alpha}{2}}_x} \|u\|^{k+2}_{L^\infty_t L^{k+2}_y} \frac{1}{R^{1/4}},
\end{align}
leading to
\begin{equation}\label{Z_23_final}
    \eqref{Z_23} = \eqref{Z_231} + O_{A}\left(\frac{1}{R^{1/4}}\right).
\end{equation}

The term $\eqref{Z_24}$, since it contains terms with $\mathcal HD^{\frac{\alpha}{2}}$, is treated in a completely analogous way, leading to
\begin{align}\label{Z_24_final}
    \eqref{Z_24}= \eqref{Z_231} + O_{A}\left(\frac{1}{R^{1/4}}\right).
\end{align}
Turning to \eqref{Z_25}, we have
\begin{align}
    \eqref{Z_25} &\label{Z_251}\tag{$Z_{2,5,1}$}= -\frac{k+1}{k+2} \iiint_{x,y,\ell} \varphi_{R,\ell} u(y)D^{\alpha}_y(\varphi_{R,\ell}u)(y)\,(\varphi_{R,\ell} u)^{k+2}(x) \\
    &\label{Z_252}\tag{$Z_{2,5,2}$}\quad+\frac{k+1}{k+2}\iiint_{x,y,\ell} \varphi_{R,\ell}u(x)[D^{\alpha}_x;\varphi_{R,\ell}] u(x) 
    \varphi_{R,\ell}^{k+2}\,u^{k+2}(y) \\
    &\label{Z_253}\tag{$Z_{2,5,3}$}\quad- \frac{k+1}{k+2}\iiint_{x,y,\ell} \varphi_{R,\ell}^2u(x)D^{\alpha}_xu(x) (\varphi_{R,\ell}^2 - \varphi_{R,\ell}^{k+2})(y)\,u^{k+2}(y).
\end{align}
\noeqref{Z_251}   
The commutator in \eqref{Z_252} yields the bound
\begin{equation}
    |\eqref{Z_252}|\lesssim_A\frac{1}{R^{1/4}},
\end{equation}
while, for \eqref{Z_253}, we rely on the support properties of $\varphi_R$ to write
\begin{equation}
    \int_{\ell} [\varphi^2_R(\ell)-\varphi^{k+2}_R(\ell)] = \int_{R}^{R + R^{3/4}} [\varphi^2_R(\ell)-\varphi^{k+2}_R(\ell)] \, \frac{d\ell}{R} \lesssim \frac{1}{R^{1/4}},
\end{equation}
thus obtaining
\begin{equation}
    |\eqref{Z_253}| \lesssim_A \frac{1}{R^{1/4}},
\end{equation}
and
\begin{equation}\label{Z_25_final}
    \eqref{Z_25} = \eqref{Z_251}+ O_{A}\left(\frac{1}{R^{1/4}}\right).
\end{equation}
Finally, for \eqref{Z_26},
\begin{align}
    \eqref{Z_26}& = -\frac{2(k+1)}{(k+2)^2} \iiint_{x,y,\ell}  \varphi_{R,\ell}^2(y)\, u^{k+2}(y) \,\varphi_{R,\ell}^2(x)\, u^{k+2}(x) \\
    &\label{Z_261}\tag{$Z_{2,6,1}$}=-\frac{2(k+1)}{(k+2)^2} \iiint_{x,y,\ell}  (\varphi_{R,\ell} u)^{k+2}(y) \,(\varphi_{R,\ell} u)^{k+2}(x) \\
    &\label{Z_262}\tag{$Z_{2,6,2}$}\,\,\quad - \frac{2(k+1)}{(k+2)^2}\iiint_{x,y,\ell} (\varphi_{R,\ell}^2-\varphi_{R,\ell}^{k+2})(y)\,u^{k+2}(y) \varphi_{R,\ell}^{k+2}(x)u^{k+2}(x) \\
    &\label{Z_263}\tag{$Z_{2,6,3}$}\,\, \quad - \frac{2(k+1)}{(k+2)^2}\iiint_{x,y,\ell} \varphi_{R,\ell}^2(y)\,u^{k+2}(y) (\varphi_{R,\ell}^2 - \varphi_{R,\ell}^{k+2})(x)\,u^{k+2}(x).
\end{align}
\noeqref{Z_261}
The support properties of $\varphi_R$ let us write again
\begin{equation}
        |\eqref{Z_262}|+|\eqref{Z_263}| \lesssim_A \frac{1}{R^{1/4}}
\end{equation}
to conclude
\begin{equation}\label{Z_26_final}
    \eqref{Z_26} = \eqref{Z_261} +  O_{A}\left(\frac{1}{R^{1/4}}\right).
\end{equation}



Therefore, by \eqref{Z_2e_final}, \eqref{Z_21_final}, \eqref{Z_22_final}, \eqref{Z_23_final}, \eqref{Z_24_final}, \eqref{Z_25_final}, and \eqref{Z_26_final},
\begin{equation}
    Z_2(t) = -\iiint_{x,y,\ell} e[\varphi_{R,\ell}u](y)j[\varphi_{R,\ell}u](x)+O_{A}\left(\frac{1}{R^{1/4}}\right).
\end{equation}
Finally, arguing in the same fashion as in the estimate of \eqref{0.103}, we have
$$
\left|
\iiint_{x,y,\ell}
e[\varphi_{R,\ell}u](y)\,
j[\varphi_{R,\ell}u](x)
\right|
\lesssim_A 1,
$$
therefore
\begin{align}
     Z_2(t) &= -\iiint_{x,y,\ell} e[\varphi_{R,\ell}u](y)j[\varphi_{R,\ell}u](x)+O_{A}\left(\frac{1}{R^{1/4}}\right)\\
     &=-(1-\varepsilon)\iiint_{x,y,\ell} e[\varphi_{R,\ell}u](y)j[\varphi_{R,\ell}u](x) - \varepsilon\iiint_{x,y,\ell} e[\varphi_{R,\ell}u](y)j[\varphi_{R,\ell}u](x)+O_{A}\left(\frac{1}{R^{1/4}}\right)\\
     &=-(1-\varepsilon)\iiint_{x,y,\ell} e[\varphi_{R,\ell}u](y)j[\varphi_{R,\ell}u](x) +O_{A}\left(\varepsilon+\frac{1}{R^{1/4}}\right).
\end{align}
This finishes/completes the proof.
\qed
\subsection{Closure}
Finally, in this section we complete the proof of our main result by showing that the critical solution constructed in Section~\ref{sec-6} does not exist.
\begin{teo}
There does not exist a solution $u$ to \eqref{DGBO} satisfying the conclusion of Proposition \ref{prop:nontrivial_localization}.
\end{teo}
\begin{proof}
Suppose $u\in C_t H^{\frac{\alpha}{2}}_x(\R \times \R)$ is a solution to \eqref{DGBO} such that there exist  $\tilde R(u)>0$ and $\varepsilon_0(u)>0$ independent on time, and a function $x: \R_+ \to \R$, such that  \begin{equation}\label{eq:nontrivial_localization_closure}
		\inf_{t \geq 0}\;\int_{|x-x(t)|\leq \tilde R(u)} |u(t,x)|^{k+2} \, dx \geq \varepsilon_0(u)>0.
	\end{equation}

Let $R \geq 2\tilde R(u)$ and let $Z_R(t)$ be the localized interaction functional defined in
\eqref{Z_def}. By the Localization Lemma (Lemma \ref{localization_lemma}) we have
$$
\sup_{t\ge0} |Z_R(t)| \lesssim_A R,
$$
On the other hand, by Lemma~\ref{lem_Z1_bound}, Lemma~\ref{lem_Z2_bound} and Theorem~\ref{teo_monotonicity}, we have
$$
Z_R'(t) \geq \int_{|\ell|\leq R} \int_{|x-x(t)|\leq 2R } |u(x+\ell)|^{2k+2}\, dx \, \frac{d\ell}{R} \geq \varepsilon_0(u) +O_A\left(\varepsilon+\frac{C_\varepsilon}{R^{1/4}}\right)
\quad \text{for all } t\ge0.
$$

By choosing $0<\varepsilon\ll_A \varepsilon_0(u)$ and increasing $R$ once, if necessary, we conclude $ t \lesssim_u |Z_R(t)|+|Z_R(0)| \lesssim R$ for all $t \geq 0$, which is a contradiction.
\end{proof}

\noindent \textbf{Declarations:}

\noindent \textbf{Ethics approval and consent to participate:}
Not applicable.


\noindent \textbf{Availability of data and materials:}
Not applicable (this manuscript does not report data generation or analysis).

\noindent \textbf{Conflicts of interest/Competing interests:}
The authors have no conflicts of interest to declare that are relevant to the content of this article.

\noindent \textbf{Author's contributions:} The authors Campos, L.,  Linares, F. and Santos, T.S.R.  wrote and reviewed the manuscript.

\bibliographystyle{abbrv}
\bibliography{biblio.bib}

{{
  \footnotesize
\vspace{1cm}
L. Campos, \textsc{Departamento de Matemática, UFMG, Belo Horizonte, MG, Brazil.} \textit{E-mail address:} \texttt{luccas@ufmg.br}.

 F. Linares, \textsc{IMPA, Instituto de Matemática Pura e Aplicada, Rio de Janeiro, RJ, 22460-320, Brazil.} \textit{E-mail address:} \texttt{linares@impa.br}.
    
T. S. R. Santos, \textsc{Instituto de Matemática, Estatística e Computação Científica (IMECC), UNICAMP, Campinas, SP, 13083-859, Brazil.} \textit{E-mail address:} \texttt{thyagosr@unicamp.br}.
}}

\end{document}